\documentclass[12pt]{amsart}
\usepackage{amsfonts,amssymb,amscd,xy}
\usepackage[leqno]{amsmath}
\usepackage{mathrsfs}
\usepackage{amssymb}
\usepackage{amsmath}
\usepackage{amsxtra}
\usepackage{amsthm}
\usepackage{hyperref}

\newcommand{\Hom}{{\mathrm{Hom}}}

\def\s{\mathscr }

\DeclareMathAlphabet{\mathbbmsl}{U}{bbm}{m}{sl}

\newcommand{\res}{{\rm res}}
\newcommand{\tr}{{\rm tr}}

\def\uxs{U_{\be X\be/\lbe S}\le}
\def\uxss{U_{\be S}(X\le)}
\def\uyss{U_{\be S}(Y\le)}

\def\pic{{\rm{Pic}}\,}

\def\bg{{\mathbb G}}
\def\g{\varGamma}

\def\s{\mathscr }

\newcommand{\isoto}{\overset{\!\sim}{\to}}

\def\g{\varGamma}

\def\br{{\rm{Br}}\e}
\def\brp{{\rm{Br}}^{\le\prime}\lbe}

\def\xb{\overline{X}}
\def\yb{\overline{Y}}
\def\kb{\overline{k}}

\def\uys{U_{\le Y/S}\le}

\def\ugs{U_{G/S}\le}

\def\picxs{{\rm Pic}_{X\! /\be S}}

\def\npic{{\rm N}\lle\pic\lbe }

\def\bxs{\brp(\lbe X\be/\be S\le)}
\def\bys{\brp(\le Y/S\e)}

\usepackage[usernames,dvipsnames]{xcolor}
\usepackage{tabularx}
\usepackage{epsfig,amssymb}
\usepackage{bm} %%% to write bold math symbols
\usepackage{verbatim} %%to use \begin{comment}

\usepackage{hyperref}
\definecolor{labelkey}{rgb}{1,0,0}

\usepackage{verbatim} %%for verbatim

\makeatletter

\DeclareMathAlphabet{\mathcalligra}{T1}{calligra}{m}{n}

\numberwithin{equation}{section}

\newcommand{\sh}{\kern -.4em\phantom{a}^{\mathbf{\sim}}}

\newcommand{\lra}{\longrightarrow}

\newcommand{\et}{{\rm {\acute et}}}

\newcommand{\fppf}{{\rm fppf}}

\newcommand{\ks}{k^{\le\rm s}}

\newcommand{\xbar}{\overline{X}}

\def\be{\kern -.1em}
\def\le{\kern 0.03em}
\def\lle{\kern 0.015em}
\def\lbe{\kern -.025em}

\newcommand{\Z}{{\mathbb Z}}
\newcommand{\Q}{{\mathbb Q}}

\newcommand{\spec}{\mathrm{ Spec}\,}

\newcommand{\krn}{\mathrm{Ker}\e}
\newcommand{\img}{\mathrm{Im}\e}
\newcommand{\cok}{\mathrm{Coker}\,}

\newcommand{\ext}{\mathrm{Ext}}

\def\e{\kern 0.08em}

\newcommand{\xs}{X^{\rm{s}}}
\newcommand{\ys}{Y^{\le\rm{s}}}

\newtheorem{lemma}{Lemma}[section]
\newtheorem{theorem}[lemma]{Theorem}
\newtheorem{proposition-definition}[lemma]{Proposition-Definition}
\newtheorem{corollary}[lemma]{Corollary}
\newtheorem{proposition}[lemma]{Proposition}
\theoremstyle{definition}
\newtheorem{definition}[lemma]{Definition}

\theoremstyle{remark}
\newtheorem{remark}[lemma]{Remark}
\newtheorem{remarks}[lemma]{Remarks}
\newtheorem{example}[lemma]{Example}
\newtheorem{examples}[lemma]{Examples}

\begin{document}

\input xy     %%for diagrams
\xyoption{all}%%

\title[Units and Picard groups of a product]{On the group of units and the Picard group of a product}

\author{Cristian D. Gonz\'alez-Avil\'es}
\address{Departamento de Matem\'aticas, Universidad de La Serena, Cisternas 1200, La Serena 1700000, Chile}
\email{cgonzalez@userena.cl}
%\thanks{Partially supported by Fondecyt grant 1160004.}
\date{\today}


\maketitle

\topmargin -1cm

\smallskip

\maketitle

\topmargin -1cm

\begin{abstract} Let $S$ be a reduced scheme and let $f\colon X\to S$ and $g\colon Y\to S$ be faithfully flat morphisms locally of finite presentation with reduced and connected maximal geometric fibers. We discuss the canonical maps $\bg_{m,\le S}\lbe(\lbe X\le)\oplus \bg_{m,\le S}\lbe(\le Y\le)\to\bg_{m,\le S}\lbe(X\lbe\times_{\be S}\lbe Y\le)$ and $\pic X\e\oplus\e\pic Y\to \pic\be(X\lbe\times_{\be S}\lbe Y\e)$ induced by $f$ and $g$. Under certain additional conditions on $S,X,Y,f$ and $g$, we describe the kernel and cokernel of the preceding maps, thereby extending known results when $S$ is the spectrum of a field.
\end{abstract}

\section{Introduction}

In this paper a {\it $k$-variety}, where $k$ is a field, is a geometrically integral $k$-scheme of finite type.

Let $k$ be a field and let $X$ be a $k$-variety. If $k$ has characteristic 0 and $\overline{k}$ is an algebraic closure of $k$, Borovoi and van Hamel introduced in \cite{bvh} the ``extended Picard complex of $X$ over $k$", denoted ${\rm UPic}(\e\xbar\e)={\rm UPic}(\e X\be\times_{k}\be\overline{k}\e)$, and used it to investigate a certain ``elementary obstruction" for $k$-torsors $X$ under $G$, where $G$ is a smooth and connected linear algebraic $k$-group. The object ${\rm UPic}(\e\xbar\e)$ is a complex of Galois modules of length 2 whose zero-th cohomology is $U_{k}(\xbar\le)=\overline{k}\le[X]^{*}\be/\e\overline{k}^{\e*}$ and first cohomology is $\pic\xbar$. The elementary obstruction alluded to above
is a class $e(X)\in\ext^{1}({\rm UPic}(\e\xbar\e),\bg_{m,\le k})\simeq H^{1}_{\rm ab}(k,G\e)$ that vanishes if $X(k)\neq\emptyset$, where $H^{1}_{\rm ab}(k,G\e)$ is Borovoi's first abelian cohomology group. It is shown in \cite[Proposition 5.7]{bvh} that, if the abelianization map $H^{1}(k,G\e)\to H^{1}_{\rm ab}(k,G\e)$ has a trivial kernel, then $e(X)$ is the only obstruction to the existence of a $k$-rational point on $X$. This paper is motivated by the following variant of the above question. Let $S$ be a connected Dedekind scheme with function field $K$, $G$ a reductive $S$-group scheme and $X$ a generically trivial $S$-torsor under $G$, i.e., $X(K)\neq\emptyset$. Then, {\it what are the obstructions to extending a given section $\spec\be K\to X_{\be K}$ to a section $S\to X\e$}? As the proof of \cite[Theorem 4.5(c), p.~326]{ctx} suggests, it seems reasonable to expect that such obstructions lie in the abelian class group $C_{\rm ab}(G\e)$ of $G$ introduced in \cite{ga3}, which is a certain subgroup of $H^{1}_{\rm ab}(S_{\fppf},G\e)$. In order to discuss the above question (elsewhere),
we need to generalize the definition of ${\rm UPic}(\e\xbar\e)$ to an arbitrary base scheme and establish an analog of the additivity lemma
\begin{equation}\label{ky}
{\rm UPic}(\e\xb\times_{\kb} \yb\e)\simeq {\rm UPic}(\e\xb\e)\oplus {\rm UPic}(\e\yb\e)
\end{equation}
\cite[Lemma 5.1]{bvh}, where $X$ and $Y$ are smooth $k$-varieties and $\yb$ is rational. Via \cite[Lemma 6.4]{san}, \eqref{ky} is the key to connecting the Borovoi-van Hamel elementary obstruction $e(X)$ to the abelian cohomology group $H^{1}_{\rm ab}(k,G\e)$. The proof of \eqref{ky} combines Rosenlicht's additivity theorem
$U_{k}(\e\xb\le\times_{\kb}\le \yb\e)\simeq U_{k}(\xb\le)\oplus\e U_{k}(\yb\le)$ \cite[Theorem 2]{ros} and a lemma of Colliot-Th\'el\`ene and Sansuc \cite[Lemma 11, p.~188]{cts77} which states that $\pic\be(\e\xb\times_{\kb}\yb\e)\simeq \pic\xb\le\oplus\le \pic\yb$ if $X$ and $Y$ are smooth $k$-varieties and $\yb$ is rational.

Thus in this paper we are concerned with generalizing the indicated additivity statements of Rosenlicht and Colliot-Th\'el\`ene-Sansuc. Our generalizations yield statements that we believe are of independent interest (e.g., Corollaries \ref{col} and \ref{2} below), as well as the generalizations of ${\rm UPic}(\e\xb\e)$ and \eqref{ky}
alluded to above (these can be found in \cite[\S3 and Proposition 4.4]{ga2}, respectively). They also yield the applications to the Brauer group discussed in \cite{ga2}.

We now observe that Rosenlicht's theorem (suitably extended) is equivalent to the following statement due to Conrad \cite{con}: if $X$ and $Y$ are geometrically connected and geometrically reduced schemes locally of finite type over a field $k$, then there exists a canonical exact sequence of abelian groups
\begin{equation}\label{c1}
1\to k^{\le *}\to \bg_{m,\e k}\lbe(\lbe X\le)\oplus \bg_{m,\e k}\lbe(\le Y\le)\to\bg_{m,\e k}\lbe(X\!\times_{ k}\! Y\le)\to 1.
\end{equation}
In particular, every unit on $X\!\times_{ k}\! Y$ is a product of pullbacks of a unit on $X$ and a unit on $Y$. The generalized Rosenlicht additivity theorem established in this paper (see Corollary \ref{ros0} below) yields the following generalization of Conrad's statement:

\begin{theorem} {\rm (= Corollary \ref{prec})} \label{1} Let $S$ be a reduced scheme and let $f\colon X\to S$ and $g\colon Y\to S$ be faithfully flat morphisms locally of finite presentation with reduced and connected maximal geometric fibers. Assume that
$f\be\times_{\be S} g\e\colon X\lbe\times_{\lbe S}\lbe Y\to S$ has an \'etale quasi-section (e.g., is smooth).
Then there exists a canonical exact sequence of abelian groups
\[
1\to \bg_{m,\le S}\le(S\le)\to \bg_{m,\le S}\le(X\le)\oplus \bg_{m,\le S}\le(\le Y\le)\to\bg_{m,\le S}\le(X\!\times_{\lbe S}\be Y\le)\to \krn\pic f\le\cap\le \krn\pic g\to 0,
\]
where $\pic\be f\colon \pic S\to\pic X$ and $\pic g\colon \pic S\to\pic Y$ are the canonical maps and the intersection takes place inside $\pic S$.
\end{theorem}

To state the following corollary of the theorem, we need a definition. If $f\colon X\to S$ is a morphism of schemes, the \'etale index $I(\le f\le)$ of $f$ is the greatest common divisor
of the degrees of all finite and \'etale quasi-sections of $f$ of constant degree, if any exist. Note that $I(\le f\le)$ is defined (and is equal to 1) if $f$ has a section. Further, if $I(\le f\le)$ is
defined, then it annihilates the subgroup $\krn\pic f$ of $\pic S$. Theorem \ref{1} and work of Tamagawa \cite{tam} yield the following statement.

\begin{corollary}\label{col} Let $S$ be a smooth, affine and geometrically connected curve over a finite field and let $f\colon X\to S$ and $g\colon Y\to S$ be smooth and surjective morphisms with geometrically irreducible generic fibers. Then $h_{S}=|\e\pic S\e|$ is finite and both
$I(\le f\e)$ and $I(\e f\e)$ are defined. Assume that ${\rm gcd}(I(\le f\e), I(\e f\e), h_{S})=1$. Then every unit on $X\times_{S}Y$ is a product of pullbacks of a unit on $X$ and a unit on $Y$.
\end{corollary} 

\medskip

If $S=\spec k$, where $k$ is a field, the canonical map $\pic X\lbe\oplus\e\pic Y\to \pic\be(X\be\times_{\lbe S} Y\e)$ was discussed by Ischebeck \cite[\S1]{isch}, Colliot-Th\'el\`ene and Sansuc \cite[Lemma 11, p.~188]{cts77} and Sansuc \cite[Lemma 6.6(i), p.~40]{san}. Over a more general base scheme $S$, we obtain

\begin{theorem} {\rm (= Corollary \ref{non3})}  Let $S$ be a locally noetherian normal scheme and let $f\colon X\to S$ and $g\colon Y\to S$ be faithfully flat morphisms locally of finite type. Assume that
\begin{enumerate}
\item[(i)]  $X$, $Y$ and $X\be\times_{S}\be Y$ are locally factorial,
\item[(ii)] for every point $s\in S$ of codimension $\leq 1$, the fibers $X_{\lbe s}$ and $Y_{\be s}$ are geometrically integral,
\item[(iii)] the \'etale index $I(\le f\!\times_{\be S}\be g\le)$ is defined and is equal to $1$, and
\item[(iv)] $\pic\be(X_{\eta}^{\le\rm s}\be\times_{k(\eta)^{\rm s}}\be Y_{\!\eta}^{\e\rm s}\le)^{\g(\eta)}=(\e\pic X_{\eta}^{\le\rm s})^{\g(\eta)}\be\oplus\be(\e\pic Y_{\!\eta}^{\e\rm s})^{\g(\eta)}$ for every maximal point $\eta$ of $S$, where $\g(\eta)={\rm Gal}(k(\eta)^{\rm s}/k(\eta))$.
\end{enumerate}	
Then there exists a canonical exact sequence of abelian groups
\[
0\to\pic S\to\pic X\be\oplus\be\pic Y\to \pic\be(X\!\times_{\be S}\! Y\e)\to 0.
\]
\end{theorem}

We remark that hypothesis (iv) of the theorem holds in many cases of interest, e.g., if either $X_{\eta}^{\le\rm s}$ or $Y_{\!\eta}^{\le\rm s}$ is rational or if $X_{\eta}$ and $Y_{\!\eta}$ are smooth and projective and there exist no nonzero $k(\eta)$-homomorphisms between their corresponding Picard varieties. See Examples \ref{sav1} and \ref{sav2} below.

\smallskip

We now observe the following corollary of the theorem and of the preceding remark.

\begin{corollary}\label{2} Let $S$ be a connected Dedekind scheme with function field $K$ and let $G$ and $H$ be smooth $S$-group schemes with connected fibers. Assume that one of the following conditions holds:
\begin{enumerate}
\item[(i)] either $G_{\lbe K}$ or $H_{\lbe K}$ is rational over a separable closure of $K$ (e.g., reductive), or
\item[(ii)] $G_{\lbe K}\!$ and $H_{\lbe K}$ are abelian varieties such that $\Hom_{\lle K}\lbe(G_{\lbe K},H_{\lbe K})=0$.
\end{enumerate}
Then there exists a canonical exact sequence of abelian groups
\[
0\to\pic S\to\pic G\be\oplus\lbe\pic H\to \pic\be(\le G\!\times_{\lbe S}\!\le H\e)\to 0.
\]
In particular, every line bundle on $G\!\times_{\lbe S}\!\le H$ is isomorphic to a product of pullbacks of a line bundle on $G$ and a line bundle on $H$.
\end{corollary}

\smallskip

The paper is organized as follows. In the preliminary Section \ref{two} we introduce the (\'etale) sheaf of relative units $\uxs$ and collect several basic facts about this and other objects. The (unavoidably) technical Section \ref{3} relates the \'etale cohomology groups of the multiplicative group scheme of a product to the corresponding objects for each of the factors, introducing in particular a number of functors to be discussed later in the paper. In Section \ref{four} we prove the generalized Rosenlicht additivity theorem. The proof of the key particular case of this theorem, namely Theorem \ref{ros1}, relies on the exactness of Conrad's sequence \eqref{c1} and on a density argument of Raynaud. In Section \ref{5} we discuss the  Picard group, using as a key ingredient a result of Raynaud that determines conditions under which a generically trivial invertible sheaf on an $S$-scheme $X$ is the pullback of an invertible sheaf on $S$ (see Proposition \ref{ray1}).

\section*{Acknowledgements}
I thank Mikhail Borovoi, K\c{e}stutis \v{C}esnavi\v{c}ius, Brian Conrad and Dino Lorenzini for helpful comments, Cyril Demarche for sending me an argument used in the proof of Theorem \ref{ros1}, Mathieu Florence for suggesting the restriction-corestriction argument used in the proof of Proposition \ref{kad} and C\'edric P\'epin for sending me the proof of Corollary \ref{ray2}. I also thank both referees of this paper for a thorough and speedy review process and the referee of \cite{ga2} for suggesting, via Example \ref{sav2}(a) below, that my original rationality hypotheses could be considerably relaxed (at least for projective varieties). Finally, I thank Chile's Fondecyt for financial support via grants 1120003 and 1160004.

\section{Preliminaries}\label{two}

If $A$ is an object of a category, the identity morphism of $A$ will be denoted by $1_{\be A}$. The category of abelian groups will be denoted by $\mathbf{Ab}$. An {\it exact and commutative} diagram in an abelian category is a commutative diagram with exact rows and columns.

\smallskip

The following lemma will be applied several times in this paper.

\begin{lemma}\label{ker-cok} If $f$ and $g$ are morphisms in an abelian category $\mathcal A$ such that $g\be\circ\!\be f$ is defined, then there exists a canonical exact sequence in $\mathcal A$ 
\[
0\to\krn f\to\krn\lbe(\e g\be\circ\!\be f\e)\to\krn g\to\cok f\to\cok\be(\e g\be\circ\!\be f\e)\to\cok g\to 0.
\]
\end{lemma}
\begin{proof} See, for example, \cite[Hilfssatz 5.5.2, p.~45]{bp}.
\end{proof}

Let $S$ be a scheme. If $S$ is clear from the context and $X$ is an $S$-scheme, $X\be\times_{S}\be T$ will be denoted by $X_{\lbe T}$. If $f\colon X\to Y$ is an $S$-morphism of schemes, we will write $f_{\lbe T}$ for the $T$-morphism of schemes $f\times_{S}1_{\lbe T}\colon X_{\lbe T}\to Y_{ T}$. If $h\colon Z\to \bg_{m,\e S}$ is a morphism of $S$-schemes, we will write $h^{-1}$ for the composition of $h$ with the inversion automorphism of $\bg_{m,\e S}$.

\smallskip

Consider the following diagram in the category of schemes
\[
\xymatrix{X_{T}\ar[d]\ar[r]& X\ar[d]\\
T\ar[r]& S.}
\]
Let $Y$ be any $S$-scheme. The universal property of the fiber product yields canonical bijections
\[
\xymatrix{\Hom_{\le X}(\le X_{T},Y_{\be X})\ar[dr]_(.45){\simeq}\ar[rr]^{\Phi}_{\simeq}&&\Hom_{\e T}\le(\le X_{T},Y_{T}) \ar[dl]^(.45){\simeq}\\
& \Hom_{\le S}\le(\le X_{T},Y)&}
\]
where the left (respectively, right)-hand oblique map is induced by composition with the first projection morphism $Y_{\be X}\to Y$ (respectively, $Y_{T}\to Y\e$) and the horizontal arrow is defined so that the triangle commutes. Taking $Y=\bg_{m,\le S}$ above, we obtain an isomorphism of abelian groups
\begin{equation}\label{cb}
\Phi\colon \Hom_{\le X}(\le X_{\lbe T},\bg_{m,\le X})\overset{\!\sim}{\to} \Hom_{\e T}\le(\le X_{\lbe T},\bg_{m,\le T})
\end{equation}
which maps an $X$-morphism $\beta\colon X_{T}\to \bg_{m,\le X}$ to the unique $T$-morphism $\Phi(\le\beta)\colon X_{T}\to \bg_{m,\le T}$ such that the diagram 
\[
\xymatrix{X_{T}\ar[d]_{\Phi(\le\beta)}\ar[r]^(.45){\beta}& \bg_{m,\le X}\ar[d]\\
\bg_{m,\le T}\ar[r]& \bg_{m,\le S}}
\]
commutes (the unlabeled maps above are projections onto first factors).

We also need to recall that, if $G$ is an $S$-group scheme and $Z\to T\to S$ are morphisms of schemes, then there exists a canonical isomorphism of abelian groups
\begin{equation}\label{bc2}
G_{T}(Z\le)=G(Z\le),
\end{equation}
where on the left $Z$ is regarded as a $T$-scheme via $Z\to T$ and, on the right, $Z$ is regarded as an $S$-scheme via the composition $Z\to T\to S$.

If $T\to S$ is a morphism of schemes, we will make the identification
\begin{equation}\label{gid}
\bg_{m,\le S}(T\le)=\Hom_{S}\le(\le T,\bg_{m,\le S})=\g(T,\mathcal O_{T}^{\e *}).
\end{equation}

Now let $S_{\et}$ denote the small \'etale site over $S$. The category of abelian sheaves on $S_{\et}$ will be denoted by $S_{\et}^{\le\sim}$. If $F\in S_{\et}^{\le\sim}$ and $g\colon T\to S$ is an \'etale morphism, then the inverse image sheaf $g^{*}F\in T_{\et}^{\e\sim}$ is given by
\begin{equation}\label{inv}
(\e g^{*}\be F\e)(Z)=F(Z\le)
\end{equation}
for any \'etale morphism $Z\to T$, where, on the right, $Z$ is regarded as an $S$-scheme via the composite morphism $Z\to T\overset{\!g}{\to} S$. See, e.g., \cite[p.~89]{t}. If $G$ is an $S$-group scheme, then \eqref{bc2} and \eqref{inv} yield the following equality in $T_{\et}^{\e\sim}$: 
\begin{equation}\label{inv2}
g^{*}G=G_{T}.
\end{equation}

\medskip

If $k$ is a field, we will write $\ks$ for a fixed separable algebraic closure of $k$ and $\g=\mathrm{Gal}\le(\ks/k)$ for the corresponding absolute Galois group. If $S$ is a scheme and $\eta$ is a maximal point of $S$, we will write $\g(\eta)=\mathrm{Gal}\le(k(\eta)^{\rm s}\lbe/\le k)$. 

\smallskip

Let $k$ be a field and let $X$ be a $k$-scheme such that $X\lbe(k^{\e\prime}\le)\neq \emptyset$ for some finite subextension $k^{\e\prime}\be/k$ of $\ks\be/k$. This is the case, for example, if $X$ is geometrically reduced and locally of finite type over $k$ (see \cite[\S3.2, Proposition 2.20, p.~93]{liu} for the finite type case and note that the more general case can be obtained by applying the indicated reference to a nonempty open affine subscheme of $X$).

If $X$ is as above, the {\it separable index} of $X$ over $k$ is the integer
\begin{equation}\label{ind}
I(X)=\gcd\{[k^{\e\prime}\be\colon\be k\e]\colon \text{$k^{\e\prime}\!/\le k\subset \ks\be/k$ finite and $X\!\left(k^{\e\prime}\e\right)\neq \emptyset$}\}.
\end{equation}

\smallskip

Recall that, if $X$ is a topological space, the {\it maximal points} of $X$ are the generic points of the irreducible components of $X$. If $f\colon X\to S$ is a morphism of schemes, a {\it maximal fiber of $f$} is a fiber of $f$ over a maximal point of $S$.

\begin{lemma}\label{mf} Let $f\colon X\to S$ and $g\colon T\to S$ be morphisms of schemes. If $g$ is faithfully flat and locally of finite presentation, then the maximal fibers of $f_{\le T}\colon X_{T}\to T$ are geometrically connected (respectively, geometrically reduced) if, and only if, the maximal fibers of $f$ are geometrically connected (respectively, geometrically reduced).
\end{lemma}
\begin{proof} By \cite[${\rm IV}_{2}$, Theorem 2.4.6]{ega}, $g$ is open. Thus, by \cite[${\rm IV}_{1}$, Corollary 1.10.4 and its proof]{ega}, $g$ induces a surjection from the set of maximal points of $T$ to the set of maximal points of $S$. Now, if $\eta$ is a maximal point of $S$ and $\eta^{\e\prime}$ is a maximal point of $T$ lying over $\eta$, then $X_{\eta}\be\times_{\eta}\be\eta^{\e\prime}=X_{T}\be\times_{T}\be\eta^{\e\prime}$, whence the lemma follows.
\end{proof}

Recall that a morphism of schemes $f\colon X\to S$ is called {\it schematically dominant} if the canonical homomorphism $f^{\#}\colon \s O_{\be S}\to f_{*}\s O_{\be X}$ of Zariski sheaves on $S\e$ is injective \cite[\S5.4]{ega1}. For example, if $S$ is reduced and $f\colon X\to S$ is dominant (e.g., $S=\spec k$, where $k$ is a field), then $f$ is schematically dominant \cite[Proposition 5.4.3, p.~284]{ega1}.

\begin{lemma} \label{sor} The following holds.
\begin{enumerate}
\item[(i)] A faithfully flat morphism is schematically dominant.
\item[(ii)] If $f\colon X\to S$ is schematically dominant and $T\to S$ is flat and  locally of finite presentation, then $f_{T}\colon X_{T}\to T$ is schematically dominant.
\item[(iii)] If $f\colon X\to S$ has a section, then $f$ is schematically dominant.
\end{enumerate}
\end{lemma}
\begin{proof} For (i), see, e.g., \cite[Proposition 52, p.~10]{pic}. For (ii), see \cite[$\textrm{IV}_{3}$, Theorem 11.10.5(ii)]{ega}. To prove (iii), let $\sigma\colon S\to X$ be a section (i.e., right inverse) of $f$. Then $f_{\lbe *}(\sigma^{\#})\colon f_{\lbe *}\s O_{\be X}\to f_{\lbe *}\sigma_{*}\s O_{\be S}=\s O_{\be S}$ is a retraction (i.e., left inverse) of $f^{\#}\colon \s O_{\be S}\to f_{*}\s O_{\be X}$. The injectivity of $f^{\#}$ follows.
\end{proof}

Let $f\colon X\to S$ be a morphism of schemes and let
\begin{equation}\label{fb}
f^{\e\flat}\colon \bg_{m,S}\to f_{\lbe *}\bg_{m,X}
\end{equation}
be the canonical morphism of abelian sheaves on $S_{\et}$ induced by $f$. Thus, if $T\to S$ is \'etale, then $f^{\e\flat}\lbe(T\le)$ is the map
\[
f^{\e\flat}\lbe(T\le)\colon \Hom_{S}(\le T,\bg_{m,S})\to \Hom_{X}(\le X_{T},\bg_{m,X}),\, c\mapsto c_{\lbe X}.
\]
If we identify $\Hom_{X}(\le X_{T},\bg_{m,X})$ and $\Hom_{T}(\le X_{T},\bg_{m,T})$ via the isomorphism \eqref{cb}, then $f^{\e\flat}\lbe(T\le)$ is identified with the homomorphism
\[
f^{\e\flat}(T\le)\colon \Hom_{S}(\le T,\bg_{m,S})\to \Hom_{\e T}(\le X_{T},\bg_{m,\le T})
\]
which maps $c\in \Hom_{S}(\le T,\bg_{m,S})$ to the $T$-morphism $X_{T}\to \bg_{m,T}=\bg_{m,S}\be\times_{S}\be T$ whose first component $X_{T}\to \bg_{m,S}$ factors as $X_{T}\to T\overset{\!c}{\to}\bg_{m,S}$. In particular,
if $c\in \Hom_{S}(\le S,\bg_{m,S})$ is a section of $\bg_{m,S}$ over $S$, then
$f^{\e\flat}\be(S\le)(c)\in \Hom_{\e S}(\le X,\bg_{m,\le S})$ is given by
\begin{equation}\label{iden}
f^{\le\flat}\be(S\le)(c)=c\lbe\circ\be f.
\end{equation}

Let $Y\overset{\!g}{\to} X\overset{\!f}{\to}S$ be morphisms of schemes. We will make the identification $(\e f\be\circ\be g)_{\lbe *}=f_{\le *}\be\circ g_{\le *}$. Then $(\e f\be\circ\be g)^{\le\flat}\colon \bg_{m,S}\to f_{\lbe *}(\e g_{*}\bg_{m,\le Y})$ factors as
\[
\bg_{m,S}\overset{f^{\le\flat}}{\to}f_{*}\bg_{m,X}\overset{f_{\lbe *}\lbe(\le g^{\le\flat}\lbe)}{\lra}f_{*}\lbe(\e g_{*}\bg_{m,Y}\lbe),
\]
i.e., 
\begin{equation}\label{marv}
f_{*}\lbe(\e g^{\flat})\be \circ\be  f^{\le\flat}=(\e f\be\circ\be g\le)^{\flat}.
\end{equation}
In particular, if $f$ has a section $\sigma\colon S\to X$, then 
\begin{equation}\label{marv2}
f_{\lbe *}\lbe(\sigma^{\le\flat})\be \circ\be  f^{\le\flat}=1_{\le\bg_{m,S}}.
\end{equation}
We now introduce the {\it \'etale sheaf of relative units on $S$}:
\begin{equation}\label{uxs2}
\hskip 1em\uxs=\cok\be f^{\e\flat}.
\end{equation}

\begin{lemma}\label{sch-dom} Let $f\colon X\to S$ be a schematically dominant morphism of schemes.
Then $f^{\e\flat}$ is an injective morphism of abelian sheaves on $S_{\et}$. Consequently, 
\[
1\to\bg_{m,\le S}\overset{\! f^{\flat}}{\to}f_{\lbe *}\bg_{m,\le X}\to\uxs\to 1
\]
is an exact sequence of \'etale sheaves on $S$. 
\end{lemma}
\begin{proof} We observe that, if $T\to S$ is flat and locally of finite presentation and $f$ is schematically dominant, then, by Lemma \ref{sor}(ii), $f_{T}\colon X_{T}\to T$ is schematically dominant as well. Consequently, if $T\to S$ is \'etale, then $f^{\e\#}_{T}\colon\mathcal O_{T}^{\e *}\to (f_{T})_{*}\mathcal O_{\be X_{T}}^{\e *}$ is an injective homomorphism of Zariski sheaves on $T$. Thus the map $\g(T,\mathcal O_{T}^{\e *})\to \g(X_{T},\mathcal O_{\be X_{T}}^{\e *})$ is injective. Via the identifications \eqref{gid}, the latter map is identified with $f^{\e\flat}\be(T\le)$, which completes the proof.
\end{proof}

\begin{lemma}\label{sd-ci} Let $f\colon X\to S$ be a morphism of schemes and let $g\colon T\to S$ be an \'etale morphism. Then $g^{\lle *}\uxs=U_{\be X_{\lbe T}\lbe/\le T}$ in $T_{\et}^{\e\sim}$. In particular, if $T^{\e\prime}\to T$ is an \'etale morphism, then $U_{\be X_{\lbe T}\lbe/\le T}\le(T^{\e\prime}\e)=\uxs(T^{\e\prime}\le)$.
\end{lemma}
\begin{proof} Let $Z\to T$ be an \'etale morphism. By \eqref{bc2}, \eqref{inv} and the identifications $T_{\be X}=T\times_{S}X=X\times_{S}T=X_{T}$, we have
\[
\begin{array}{rcl}
(\e g^{*}\be f_{\lbe *}\bg_{m,\le X})(Z\e)&=&(\e f_{\lbe *}\bg_{m,\le X})(Z\e)=\bg_{m,\le X}(Z_{X})\\
&=&\bg_{m,\le X_{T}}(Z\be\times_{T}\be X_{T})=((f_{T}\lbe)_{*}\bg_{m,\e X_{T}})(Z\e).
\end{array}
\]
Thus $g^{*}f_{\lbe *}\bg_{m,\le X}=(f_{T})_{*}\bg_{m,\e X_{T}}$ in $T_{\et}^{\e\sim}$. On the other hand,  by \eqref{inv2}, $g^{*}\e\bg_{m,\le S}=\bg_{m,\le T}$ in $T_{\et}^{\e\sim}$. Under these identifications, the map $g^{*}(f^{\le\flat})\colon g^{*}\le\bg_{m,\le S}
\to g^{*}f_{\lbe *}\bg_{m,\le X}$ is identified with $f_{T}^{\e\flat}$, whence the lemma follows.
\end{proof}

If $f\colon X\to S$ is schematically dominant and $T\to S$ is \'etale then, by Lemma \ref{sch-dom}, the sequence of \'etale sheaves on $T$
\begin{equation}\label{useq}
0\to\bg_{m,\le T}\overset{\be f_{T}^{\le\flat}}{\lra}(f_{T})_{*}\bg_{m,\le X_{T}}\to U_{\lbe X_{T}\be/\le T}\to 0
\end{equation}
is exact. Since $U_{\lbe X_{T}\be/\le T}\le(\le T\e)=\uxs(\le T\e)$ by Lemma \ref{sd-ci}, \eqref{bc2} and the preceding sequence induce an exact sequence of abelian groups
\begin{equation}\label{useq2}
0\to\bg_{m,\le S}(T\e)\to \bg_{m,\le S}(X_{T}\lbe)\to\uxs(\le T\e)\to\pic \e T\to H^{1}\lbe(T_{\et},(f_{T})_{*}\bg_{m,\le X_{T}}\be),
\end{equation}
where the last homomorphism is the map $H^{1}\be(f_{\lbe T}^{\le\flat}\be)$. We will write
\begin{equation}\label{uxss}
\uxss=\uxs(\lbe S\le).
\end{equation}
A comment on notation is in order. By \eqref{comp} and \eqref{leray0} below, the canonical map $\pic\be f\colon \pic S\to \pic X$ factors as
\[
\pic S\overset{H^{1}\be(\le f^{\le\flat})}{\lra}H^{1}(S_{\et},f_{*}\bg_{m,X})\hookrightarrow
\pic X.
\]
Thus \eqref{useq2} yields an exact sequence
\[
0\to\bg_{m,\le S}(S\e)\to \bg_{m,\le S}(X\lbe)\to\uxss\to\pic \e S\overset{\!\pic\be f}{\lra} \pic X.
\]
It follows that the group $\uxss$ \eqref{uxss} coincides with the group that is usually denoted by $\pic\Phi(\le f\le)$ in the classical literature (see, e.g., \cite[p.~22]{bm}, \cite[p.~148]{isch}).

\smallskip

If $f\colon X\to S$ is a morphism which has a section, then $f$ is schematically dominant by Lemma \ref{sor}(iii) and the preceding discussion applies to $f$. In this case \eqref{useq} is, in fact, a (non-canonically) split exact sequence of \'etale sheaves on $T$. Indeed, if $\sigma\colon S\to X$ is a section of $f$, then $\sigma_{T}\colon T\to X_{T}$ is a section of $f_{T}$ and $(f_{T})_{*}(\sigma_{T}^{\le\flat}\le)$ is a retraction of $f_{T}^{\e\flat}$ \eqref{marv2} which splits \eqref{useq}. Consequently, for every integer $r\geq 0$ and every \'etale morphism $T\to S$, \eqref{useq} induces a split exact sequence of abelian groups
\[
0\lra H^{\le r}\lbe(T_{\et},\bg_{m,\le T}\be)\overset{\! H^{r}\be(\le f_{\lbe T}^{\le\flat})}{\lra} H^{\le r}\lbe(T_{\et},(f_{T})_{*}\bg_{m,X_{T}}\be)\lra H^{\le r}\lbe(T_{\et},U_{\lbe X_{T}\be/\e T}\be)\lra 0.
\]
In particular, setting $r=0$ above, we obtain a canonical isomorphism of abelian groups
\begin{equation}\label{class}
\uxs(\le T\e)=\g(X_{T},\s O_{\be X_{T}}\lbe)^{*}/\g(T,\s O_{T}\lbe)^{*}.
\end{equation}
If $f$ does not have a section but is schematically dominant and $\pic \e T=0$ (e.g., $T=\spec k$, where $k$ is a field), then the preceding formula also holds by the exactness of \eqref{useq2}.

\begin{example}\label{exs} If $k$ is a field, $X\to\spec k$ is a $k$-scheme and $k^{\e\prime}\be/k$ is a finite subextension of $\ks\be/k$, set $X^{\prime}=X\times_{k}\spec k^{\e\prime}$ and $\g(X^{\prime}, \s O_{\be X^{\prime}}\lbe)=k^{\e\prime}[X]$. Further, write $\xs=X\times_{k}\spec\ks$ and $\ks[X]=\g(\xs, \mathcal O_{\be\xs}\lbe)$. We have $U_{X/k}(k^{\e\prime}\e)=k^{\e\prime}\le[X]^{*}\lbe/(k^{\e\prime}\e)^{*}$ by \eqref{class}. Consequently, $U_{X/k}$ is the \'etale sheaf on $k$ associated to the $\g$-module $\varinjlim_{\e k^{\e\prime}} k^{\e\prime}\le[X]^{*}\lbe/(k^{\e\prime}\le)^{*}=\ks[X]^{*}/\e (\ks)^{*}$, where $\g={\rm Gal}(\ks\be/k\le)$. See \cite[II, \S2, pp.~92-95]{t}. In particular,
\begin{equation}\label{uk}
U_{k}(X\le)=U_{X/k}(k\le)=k[X]^{*}\be/\e k^{*}.
\end{equation} 
\end{example}

\smallskip

We now return to our general discussion.

\smallskip

Let $S$ be a scheme, $\s A$ an abelian category, $\s C$ a full subcategory of $(\textrm{Sch}/S\e)$ which is stable under products and contains $1_{\lbe S}$, and $F\colon \s C\to \s A$ a contravariant functor. If $f\colon X\to S$ and $g\colon Y\to S$ are objects of $\s C$, we will write $F(\le f\le)=F(X)$ and $F(\le g\le)=F(Y)$. Let $\textrm{pr}_{1}\colon F(X)\oplus F(Y\le)\to F(X)$ and $\textrm{pr}_{2}\colon F(X)\oplus F(Y\le)\to F(Y\le)$ be the first and second projection morphisms in $\s A$. The canonical $S$-morphisms $p_{X}\colon X\times_{S}Y\to X$ and $p_{\e Y}\colon X\be\times_{S}\be Y\to Y$ induce a morphism in $\s A$
\begin{equation}\label{fxy}
\psi_{X,\e Y}=F(\e p_{\lbe X})\circ \lbe\textrm{pr}_{1}+F(\e p_{\e Y})\circ\lbe\textrm{pr}_{2} \colon F(X)\oplus F(Y)\to F(X\be\times_{S}\be Y\le)
\end{equation}
which is natural in $X\to S$ and $Y\to S$.

Now assume that $Y=G$ is a group in $\s C$, both $X\to S$ and $G\to S$ are flat and locally of finite presentation and $X\to S$ is equipped with a right action
\begin{equation}\label{act}
\varsigma\colon X\be \times_{\lbe S} G\to X
\end{equation}
of $G$ over $S$. Assume, in addition, that the map $\psi_{X,\le G}\colon F(X)\oplus F(G\le)\to F(X\be \times_{\lbe S} G\e)$ is an isomorphism and define maps $\pi_{\be X}\colon F(X\be \times_{S}\lbe G\e)\to F(X)$ and $\pi_{\lbe G}\colon F(X\be \times_{S}\lbe G\e)\to F(G\le)$ by $\pi_{\be X}=\textrm{pr}_{1}\circ \psi_{X,\le G}^{-1}$ and
$\pi_{\lbe G}=\textrm{pr}_{2}\circ \psi_{X,\le G}^{-1}$. Then the following holds
\begin{equation}\label{fpg}
F(\e p_{G})\be\circ\be \pi_{G}=1_{F(X\times_{S}\le G\le)}-F(\e p_{X})\be\circ\be \pi_{\be X}.
\end{equation}
Further, if $i_{2}\colon F(G\le)\to F(X)\oplus F(G\le)$ is given by $i_{2}(a)=(0,a)$, then the diagram
\[
\xymatrix{0\ar[r]& F(G\le)\ar[r]^(.35){i_{2}}\ar@{=}[d]& F(X)\oplus F(G\le)\ar[r]^(.6){\textrm{pr}_{1}}\ar[d]_(.45){\simeq}^(.45){\psi_{X,G}}& F(X)\ar@{=}[d]\ar[r]& 0\\
0\ar[r]& F(G\le)\ar[r]^(.4){F(\e p_{\lle G})}& F(X\be \times_{S}\be G\le)\ar[r]^(.6){\pi_{\lbe X}}&  F(X)\ar[r]& 0
}
\]
is exact and commutative. We conclude that 
\begin{equation}\label{spe}
0\to F(G\le)\overset{\!F\lbe(\le p_{\lle G}\lbe)}{\lra} F(X\be \times_{S}\be G\le)\overset{\pi_{\lbe X}}{\lra} F(X)\to 0
\end{equation}
is a split exact sequence in $\s A$.

Now let
\begin{equation}\label{vphi}
\varphi_{X,\e G}\colon F(X)\to F(G\le),
\end{equation}
be the composition
\[
F(X)\overset{\!\!\!F(\varsigma)}{\lra}F(X\be \times_{S}\be G\e)\overset{\!\pi_{G}}{\lra}F(G\le),
\]
where $\varsigma$ is the map \eqref{act}. Further, set
\begin{equation}\label{ro}
\rho=(\varsigma,p_{X}\lbe)_{\lbe S}\colon X\be \times_{S}\be G\to X\be \times_{S}\be X.
\end{equation}

\smallskip

The following lemma generalizes \cite[Lemma 6.4]{san}.

\begin{lemma}\label{add} Let $S$ be a scheme, $\s A$ an abelian category, $\s C$ a full subcategory of $(\rm{Sch}/S\e)$ which is stable under products and contains $1_{\lbe S}$ and $F\colon \s C\to \s A$ a contravariant functor. Let $X$ and $G$ be objects of $\e\s C$ which are flat and locally of finite presentation over $S$, where $G$ is a group, and assume that $X$ is a torsor under $G$ over $S$, i.e., the map \eqref{ro} is an isomorphism. Assume, furthermore, that the map $\psi_{X,\e G}$ \eqref{fxy} is an isomorphism, so that the map $\varphi_{X,\e G}$ \eqref{vphi} is defined. If 
\begin{enumerate}
\item[(i)]  $F(S\le)=F(1_{\lbe S})=0$ and
\item[(ii)] the map $\psi_{X,\e X}\colon F(X)\oplus F(X\le)\to F(X\!\times_{\be S}\! X\e)$ \eqref{fxy} is an isomorphism,
\end{enumerate}
then $\varphi_{X,\e G}\colon F(X)\to F(G\le)$ \eqref{vphi} is an isomorphism.
\end{lemma}
\begin{proof} Let $\varepsilon\colon S\to G$ be the unit section of $G$. By functoriality, the following diagram commutes
\[
\xymatrix{
F(X)\oplus F(G\le)\ar[r]^(.53){\psi_{X,G}}\ar[d]_{(1_{F\lbe(X)},\e F(\varepsilon))}& F(X\be \times_{S}\be G\le)\ar[d]^{F(\varepsilon_{\be X}\lbe)}\\
F(X)\oplus F(S)\ar[r]^(.6){\psi_{X,S}}& F(X).
}
\]
Now, by (i), the above diagram can be identified with a commutative triangle
\[
\xymatrix{
F(X)\oplus F(G)\ar[r]^(.53){\psi_{X,G}}\ar[dr]_(.5){\textrm{pr}_{1}}& F(X\be \times_{S}\be G\le)\ar[d]^{F(\varepsilon_{\be X}\lbe)}\\
& F(X),
}
\]
i.e.,
\begin{equation}\label{pix}
F(\varepsilon_{\be X}\lbe)=\textrm{pr}_{1}\circ \psi_{X,\le G}^{-1}=\pi_{\be X}.
\end{equation}
We conclude that 
\begin{equation}\label{piv}
\pi_{\be X}\be\circ\be F(\le\varsigma)=F(\le\varsigma\be\circ\be\varepsilon_{\lbe X})=F(\le 1_{X})=1_{F(X)},
\end{equation}
where $\varsigma$ is the given right action \eqref{act}.

Now let $\delta\colon F(X)\to F(X)\oplus F(X)$ be the map $\delta(c)=(c,-c)$ and define
$\lambda=\psi_{X,\le X}\circ\e\delta\colon F(X)\to F(X\be \times_{S}\lbe X\e)$. Then there exists a canonical exact and commutative diagram 
\[
\xymatrix{0\ar[r]&F(X)\ar[r]^(.35){\delta}\ar@{=}[d]&F(X)\oplus F(X)\ar[r]^(.6){+}\ar[d]^(.48){\psi_{X,X}}_{\simeq}& F(X)\ar@{=}[d]\ar[r]&0\\
0\ar[r]&F(X)\ar[r]^(.4){\lambda}&F(X\be \times_{S}\be X\le)\ar[r]^(.6){F(\Delta)}&F(X)\ar[r]& 0,
}
\]
where $\Delta\colon X\to X\! \times_{S}\! X$ is the diagonal morphism. Since the middle vertical map in the above diagram is an isomorphism by (ii) and the top row is a split exact sequence, the bottom row is a split exact sequence as well. Now consider the diagram with split exact rows 
\begin{equation}\label{d4}
\xymatrix{0\ar[r]&F(X)\ar[r]^(.37){\lambda}\ar[d]_{\varphi_{\lbe X,\le G}}&F(X\be \times_{S}\be X\le)\ar[r]^(.6){F(\Delta)}\ar[d]^{F(\le\rho)}&F(X)\ar@{=}[d]\ar[r]&0\\
0\ar[r]&F(G)\ar[r]^(.4){F(\le p_{\le G})}&F(X\be \times_{S}\be G\le)\ar[r]^(.6){F(\varepsilon_{\be X})}&F(X)\ar[r]& 0
}
\end{equation}
where $\rho$ is the map \eqref{ro} and the bottom row is the sequence \eqref{spe} (recall that $F(\varepsilon_{\be X}\lbe)=\pi_{\be X}$ \eqref{pix}). Since $\rho$ is an isomorphism by hypothesis, the middle vertical map in the above diagram is an isomorphism as well. Further, since $\rho\e\circ\e\varepsilon_{\be X}=\Delta$, the right-hand square in \eqref{d4} commutes. Thus, to complete the proof, it suffices to check that the left-hand square in \eqref{d4} commutes. Let $p_{\le 1}$ and $p_{\le 2}$ denote (respectively) the first and second projections $X\be \times_{S}\be X\to X$. Then, by \eqref{fxy},
\[
\begin{array}{rcl}
F(\e\rho)\circ\psi_{X,X}&=& F(\e\rho)\circ F(\e p_{\le 1})\be\circ\be\textrm{pr}_{1}+F(\e\rho)\be\circ\be F(\e p_{2})\be\circ\be\textrm{pr}_{2}\\
&=& F(\le\varsigma)\be\circ\be\textrm{pr}_{1}+F(\e p_{X})\be\circ\be\textrm{pr}_{2}.
\end{array}
\]
Thus, by \eqref{fpg} and \eqref{piv},
\[
\begin{array}{rcl}
F(\e\rho)\be\circ\be\lambda&=& F(\e\rho)\be\circ\be\psi_{X,X}\be\circ\be \delta=( F(\le\varsigma)\be\circ\be\textrm{pr}_{1}+F(\e p_{X})\be\circ\be\textrm{pr}_{2})\circ\delta=F(\le\varsigma)-F(\e p_{X})
\\
&=&F(\le\varsigma)-F(\e p_{X})\be\circ\be \pi_{\be X}\be\circ\be F(\le\varsigma)=(1_{F(X\times_{S}\le G\e)}-F(\e p_{X})\be\circ\be \pi_{\be X})\be\circ\be F(\le\varsigma)\\
&=&F(\e p_{G})\be\circ\be \pi_{G}\be\circ\be F(\le\varsigma)=F(\le p_{G})\be\circ\be \varphi_{X,\le G}.
\end{array}
\]
Thus the right-hand square in diagram \eqref{d4} commutes, as required.
\end{proof}

\smallskip

An {\it \'etale quasi-section} of a morphism $f\colon X\to S$ is an \'etale and surjective morphism $\alpha\colon T\to S$ such that there exists an $S$-morphism $h\colon T\to X$, i.e., $f\circ h=\alpha$. In this case, the induced morphism $(h,1_{T})_{S}\colon T\to X_{T}$ is a section of $f_{T}\colon X_{T}\to T$. By \cite[${\rm IV}_{4}$, Corollary 17.16.3(ii)]{ega}, every smooth and surjective morphism $X\to S$ has an \'etale quasi-section.

\begin{definition}\label{eti} Let $f\colon X\to S$ be a morphism of schemes which has a finite \'etale quasi-section of constant degree. Then the {\it \'etale index} $I(\le f\le)$ of $f$ is the greatest common divisor of the degrees of all finite \'etale quasi-sections of $f$ of constant degree.
\end{definition}

If $f\colon X\to S$ has a finite \'etale quasi-section $T\to S$ of constant degree $d$ (so that $I(\le f\le)$ is defined) and $Y\to S$ is any $S$-scheme, then $Y_{T}\to Y$ is a finite \'etale quasi-section of $f_{Y}$ of constant degree $d$. It follows that $I(\le f_{Y}\le)$ is also defined and divides $I(\le f\le)$. In particular, for every maximal point $\eta$ of $S$, the \'etale (i.e., separable) index of $f_{\eta}\colon X_{\eta}\to\eta$ divides $I(\le f\le)$, i.e., $I(\le f_{\eta}\le)\!\mid\! I(\le f\le)$. Note also that, if $g\colon Y\to S$ is another morphism of schemes and $\alpha\colon T\to S$ is an \'etale quasi-section of $f\!\times_{\be S}\be g\colon X\!\times_{\be S}\! Y\to S$ of constant degree $d$ with associated $S$-morphism $h\colon T\to X\!\times_{\be S}\be Y$, then $\alpha$ is also an \'etale quasi-section of $f$ and $g$ with associated $S$-morphisms $p_{X}\be\circ\be h\colon T\to X$ and $p_{\e Y}\be\circ\be h\colon T\to Y$, respectively. It follows that, if $I\lbe (\le f\be\times_{\be S}\be g\le)$ is defined, then $I\le(\le f\le)$ and $I\le(\e g\le)$ are both defined and they divide $I\lbe (\le f\!\times_{\be S}\be g\le)$. Consequently
\begin{equation} \label{genf}
I\lbe (\le f\!\times_{\be S}\be g\le)=1\implies I(\le f_{\eta}\le)=I(\le g_{\eta}\le)=1
\end{equation}
for every maximal point $\eta$ of $S$.

Note also that the degree of a finite \'etale quasi-section $\alpha\colon T\to S$ of $f\colon X\to S$ is constant over each connected component of $S$. Thus, if $S$ is connected, then the degree of $\alpha$ is constant. Henceforth, the statement {\it $f$ has a finite \'etale quasi-section of constant degree} will be abbreviated to {\it the \'etale index of $f$ is defined} (or, simply, to {\it $I(\le f\le)$ is defined\,}).
Note that the statement {\it $I(\le f\le)$ is defined and is equal to $1$}, which appears often below, clearly holds if $f$ has a section.

\begin{remark}\label{gff} The requirement that $f\colon X\to S$ have a finite \'etale quasi-section of constant degree holds in one important case of interest: let $S$ be a non-empty open affine subscheme of a proper, smooth and geometrically connected curve over a finite field. If $f\colon X\to S$ is a smooth and surjective morphism with geometrically irreducible generic fiber, then $f$ has a finite \'etale quasi-section of constant degree (since $S$ is connected), i.e., $I(\le f\le)$ is defined. See \cite[Theorem (0.1)]{tam}.
\end{remark}

\section{The \'etale cohomology of the multiplicative group}\label{3}

In this Section we relate the \'etale cohomology of $\bg_{m,\e X\be\times_{\be S}\lbe Y}$ to the \'etale cohomology of $\bg_{m,\e X}$ and $\bg_{m,\e  Y}$ for arbitrary $S$-schemes $X$ and $Y$.

\smallskip

Let $f\colon X\to S$ be a morphism of schemes. Recall that the (\'etale) {\it relative Picard functor of $X$ over $S$} is the \'etale sheaf $\picxs$ on $S$ associated to the presheaf $(\textrm{Sch}/S)\to \mathbf{Ab}, (T\to S\le)\mapsto\pic X_{T}$.
We have
\[
\picxs=R_{\et}^{\le 1}\e f_{\lbe *}\bg_{m,X}.
\]
See \cite[\S2]{klei} and/or \cite[\S8.1]{blr} for basic information on $\picxs$.
We will write $\brp X=H^{2}(X_{\et},\bg_{m,\le X}\lbe)$ for the full cohomological Brauer group of $X\e$ (its torsion subgroup $H^{2}(X_{\et},\bg_{m,\le X}\lbe)_{\rm tors}$ is often also called the cohomological Brauer group of $X$, but the latter group will play no role in this paper). The Cartan-Leray spectral sequence associated to $f$
\begin{equation}\label{leray}
H^{\le r}\be(S_{\et}, R^{\e s}\! f_{\lbe *}\bg_{m,X})\Rightarrow H^{\le r+s}(X_{\et}, \bg_{m,X})
\end{equation}
furnishes edge morphisms $e_{\be f}^{r}\colon H^{\le r}(S_{\et},f_{\lbe *}\bg_{m,X})\to H^{\le r}(X_{\et},\bg_{m,X})$ for every $r\geq 0$. On the other hand, there exists a canonical pullback map
\begin{equation}\label{rr}
f^{\le(r)}\colon H^{r}(S_{\et},\bg_{m,S})\to H^{r}(X_{\et},\bg_{m,X})
\end{equation}
defined as the composition
\begin{equation}\label{comp}
H^{r}(S_{\et},\bg_{m,S})\overset{H^{r}\be(\le f^{\le\flat})}{\lra}H^{r}(S_{\et},f_{*}\bg_{m,X})
\overset{e_{\be f}^{r}}{\lra}H^{r}(X_{\et},\bg_{m,X}),
\end{equation}
where $f^{\le\flat}$ is the map \eqref{fb}. Note that
\begin{equation}\label{fo}
f^{(0)}=f^{\e\flat}\be(S).
\end{equation}
Further, by Lemma \ref{ker-cok} applied to the pair of maps \eqref{comp}, for every integer $r\geq 0$ there exists a canonical exact sequence of abelian groups
\begin{equation}\label{mag}
0\to\krn H^{r}\be(\le f^{\le\flat})\to\krn f^{\le(r)}\to \krn\e e_{\be f}^{r}.
\end{equation}
If $Y\overset{\!g}{\to} X\overset{\!f}{\to}S$ are morphisms of schemes, then it is not difficult to check that
\[
e_{\be f\circ g}^{r}\be\circ\be H^{r}\be(\le f_{*}\be(\e g^{\flat}))=e_{\be g}^{r}\circ  H^{r}\be(\e g^{\flat})\circ e_{\be f}^{r}.
\]
Now \eqref{marv} yields the identity
\begin{equation}\label{need}
(\e f\be\circ\be g\le)^{(r)}=g^{(r)}\!\circ\! f^{(r)}.
\end{equation}
When $r=1$ (respectively, $r=2$), \eqref{rr} will be identified with (respectively, denoted by) the canonical map $\pic\be f\colon \pic S\to \pic X$ (respectively, $\brp f\colon \brp S\to \brp X\le$). The {\it relative cohomological Brauer group of $X$ over $S$} is defined by
\begin{equation}\label{relb}
\bxs=\krn\!\left[\e \brp f\colon \brp S\to \brp X\right].
\end{equation}
More generally, for every integer $r\geq 0$ we define a contravariant functor
\begin{equation}\label{kr}
\krn^{\lbe(r)}_{\be S}\colon ({\rm{Sch}}/S\e)\to \mathbf{Ab}
\end{equation} 
by setting
\[
\krn^{\lbe(r)}_{\be S}\lbe(\le X\overset{\!f}{\to}S\le)=\krn\!\left[\e f^{\le(r)}\colon H^{r}\be(S_{\et},\bg_{m,S})\to H^{r}\be(X_{\et},\bg_{m,X})\right].
\]
Thus, for $r=1$ and $r=2$, we obtain functors
\begin{equation}\label{kpic}
\krn\pic_{\! S}\colon ({\rm{Sch}}/S\e)\to \mathbf{Ab},\, (\le X\overset{\!f}{\to}S\le)\mapsto \krn\e\pic \be f,
\end{equation}
and
\begin{equation}\label{kbr}
\brp(-/S)\colon ({\rm{Sch}}/S\e)\to \mathbf{Ab},\, (\le X\overset{\!f}{\to}S\le)\mapsto \bxs.
\end{equation}
We will also consider the functor
\begin{equation}\label{ckr}
\cok^{\!(r)}_{\!\be S}\colon ({\rm{Sch}}/S\e)\to \mathbf{Ab}
\end{equation} 
defined by
\[
\cok^{\!(r)}_{\!\be S}\be(\le X\overset{\!f}{\to}S\le)=\cok\!\be\left[\e f^{\le(r)}\colon H^{r}\be(S_{\et},\bg_{m,S})\to H^{r}\be(X_{\et},\bg_{m,X})\right].
\]
We will use the standard notation for $\cok^{\!(1)}_{\!\be S}\be(\le X\overset{\!f}{\to}S\le)$, which is 
\begin{equation}\label{npic}
\npic\be(X\!/\be S\le)=\cok\pic\be f.
\end{equation}
See, e.g., \cite{w} and \cite[Definition 3.15]{isch}. Thus, \eqref{ckr} for $r=1$ is a functor
\begin{equation}\label{npicf}
\npic\be(-/\be S)\colon ({\rm{Sch}}/S\e)\to \mathbf{Ab},\, (\le X\to S\le)\mapsto \npic\be(X\be/\be S\le).
\end{equation}

Now, by \cite[p.~309, line 8]{mi1}, the Cartan-Leray spectral sequence \eqref{leray} induces an exact sequence of abelian groups
\begin{equation}\label{leray0}
0\to  H^{1}(S_{\et},f_{\lbe *}\bg_{m,X})\overset{\!e_{\be f}^{1}}{\to}\pic X\to\picxs(S\le)\to
H^{2}(S_{\et},f_{\lbe *}\bg_{m,X})\to \br X.
\end{equation}
Thus $e_{\be f}^{1}$ induces an isomorphism of abelian groups
\begin{equation}\label{sbgp}
H^{1}(S_{\et},f_{\lbe *}\bg_{m,X})\simeq\krn[\e\pic X\to\picxs(S\le)]
\end{equation}
which identifies $H^{1}(S_{\et},f_{\lbe *}\bg_{m,X})$ with a subgroup of $\pic X$. Further, by \eqref{mag} and the injectivity of $e_{\be f}^{1}$, there exists a canonical isomorphism of abelian groups
\begin{equation}\label{mag2}
\krn H^{1}\be(\le f^{\le\flat})=\krn \pic f.
\end{equation}

\begin{remarks}\label{sect} \indent
\begin{enumerate}
\item[(a)] Let $f\colon X\to S$ be a morphism of schemes, $\alpha\colon T\to S$ a finite, \'etale and surjective morphism of constant degree $d$ and $r\geq 0$ an integer. By \eqref{inv2} and \cite[IX, \S5.1]{sga4}, $\alpha$ induces restriction and corestriction (or trace) homomorphisms $\res\colon H^{r}\be(S_{\et},\bg_{m,S})\to H^{r}\be(T_{\et},\bg_{m,T})$ and $\tr\colon H^{r}\be(T_{\et},\bg_{m,T})\to H^{r}\be(S_{\et},\bg_{m,S})$ such that $\tr\circ\res\colon H^{r}\be(S_{\et},\bg_{m,S})\to H^{r}\be(S_{\et},\bg_{m,S})$
is the multiplication by $d$ map. Consider the following commutative diagram of abelian groups with exact rows
\begin{equation}\label{big}
\xymatrix{0\ar[r]& \krn f^{\le(r)}\ar[d]^{\res}\ar[r]& H^{r}\be(S_{\et},\bg_{m,S})\ar[d]^{\res}\ar[r]^(.45){f^{\le(r)}}& H^{r}\be(X_{\et},\bg_{m,X})\ar[d]^{\res}\ar[r]& \cok f^{\le(r)}\ar[d]^{\res}\ar[r]& 0\\
0\ar[r]&\krn f_{T}^{(r)}\ar[d]^{\tr}\ar[r]& H^{r}\be(T_{\et},\bg_{m,T})\ar[d]^{\tr}\ar[r]^(.45){f_{\lbe T}^{(r)}}& H^{r}\be((X_{T})_{\et},\bg_{m,X_{T}})\ar[d]^{\tr}\ar[r]&\cok f_{T}^{(r)}\ar[d]^{\tr}\ar[r]& 0\\
0\ar[r]& \krn f^{\le(r)}\ar[r]& H^{r}\be(S_{\et},\bg_{m,S})\ar[r]^(.45){f^{\le(r)}}& H^{r}\be(X_{\et},\bg_{m,X})\ar[r]&\cok f^{\le(r)}\ar[r]& 0,
}
\end{equation}
where the maps appearing on the outer columns are induced by those appearing on the inner columns using the commutativity of the top and bottom inner squares. As noted above, the second vertical composition is the multiplication by $d$ map. Now, since $X_{T}\to X$ is a finite, \'etale and surjective morphism of constant degree $d$, the third vertical composition is also the multiplication by $d$ map. We conclude that the compositions $\krn f^{\le(r)}\overset{\!\res}{\to}\krn f_{T}^{(r)}\overset{\!\tr}{\to} \krn f^{\le(r)}$ and
$\cok f^{\le(r)}\overset{\!\res}{\to}\cok f_{T}^{(r)}\overset{\!\tr}{\to} \cok f^{\le(r)}$ are the multiplication by $d$ map.
		
\item[(b)] If $F\colon ({\rm{Sch}}/S\e)\to \mathbf{Ab}$ denotes any of the functors $\cok^{\be(r)}_{\!\be S}\!$ \eqref{ckr} (in particular, $\npic\be(-/\be S\le)$ \eqref{npicf}), then $F(S\le)=F(1_{\lbe S})=0$, i.e., $F$ satisfies condition (i) of Lemma \ref{add}. This also holds for $F=\krn^{\lbe(r)}_{\!\lbe S}\!$ \eqref{kr} (in particular, for $\krn\pic_{\! S}$ \eqref{kpic} and $\br(-/\be S\le)$ \eqref{kbr}). In fact, $\krn^{\be(r)}_{\! S}\lbe(\le f\le)=0$ for every morphism $f\colon X\to S$ that has a section, as the following (well-known) remark shows. 
		
\item[(c)] If $f\colon X\to S$ has a section $\sigma\colon S\to X$, then the induced homomorphism $\sigma^{(r)}\colon H^{r}\be(X_{\et},\bg_{m,X})\to H^{r}\be(S_{\et},\bg_{m,S})$ is a retraction of $f^{\le(r)}\colon H^{r}\be(S_{\et},\bg_{m,S})\to H^{r}\be(X_{\et},\bg_{m,X})$. Consequently, $\krn^{\be(r)}_{\lbe S}\lbe(\le f\le)=\krn f^{\le(r)}=0$ for every $r\geq 0$.
		
\item[(d)] Let $f\colon X\to S$ be a morphism which has a finite \'etale quasi-section $T\to S$ of constant degree $d$, so that the \'etale index $I(\le f\le)$ of $f$ is defined. Since $f_{T}\colon X_{T}\to T$ has a section, (c) shows that $\krn f_{T}^{(r)}=0$ for every integer $r\geq 0$. It now follows from diagram \eqref{big} that 
$d\lbe\cdot\lbe\krn\le f^{\le(r)}=0$. Consequently, $\krn\le f^{\le(r)}$ is annihilated by $I(\le f\le)$ for every integer $r\geq 0$. In particular, $\krn\le \pic\be f$ and $\bxs$ \eqref{relb} are  annihilated by $I(\le f\le)$.
\end{enumerate}
\end{remarks}

Let
\begin{equation}\label{nb}
\nabla\colon\bg_{m,\le S}\to \bg_{m,\le S}\le\oplus\le \bg_{m,\le S}.
\end{equation}
be the morphism of \'etale sheaves on $S$ defined by $\nabla(T)(c)=(c\le,c^{-1})$, where $T$ is an \'etale $S$-scheme and $c\in \bg_{m,\le S}(\le T\le)$. 
For every $r\geq 0$, \eqref{nb} induces a homomorphism of abelian groups
\begin{equation}\label{delr}
\nabla^{r}=H^{r}\be(\nabla)\colon H^{r}\be(S_{\et},\bg_{m,S})\to H^{r}\be(S_{\et},\bg_{m,S})\oplus H^{r}\be(S_{\et},\bg_{m,S}),\e \xi\mapsto(\xi,\xi^{-1}\le).
\end{equation}

Let $f\colon X\to S$ and $g\colon Y\to S$ be morphisms of schemes. For every integer $r\geq 0$, the canonical projections $p_{\lbe X}\colon X\!\times_{\be S}\! Y\!\to\! X$ and $p_{\e Y}\colon X\!\times_{\be S}\! Y\!\to\! Y$ induce homomorphisms $p_{\be X}^{(r)}$ and $p_{\le Y}^{(r)}$ \eqref{rr}. Define a homomorphism
\begin{equation}\label{pxy}
p_{XY}^{\le r}\colon H^{r}\be(X_{\et},\bg_{m,X}\be)\lbe\oplus\lbe H^{r}\be(\le Y_{\et},\bg_{m,Y}\be)\to H^{r}\!\left((X\!\times_{\be S}\! Y\le)_{\et},\bg_{m,\e X\lbe\times_{\be S}\lbe Y}\lbe\right)
\end{equation}
by setting $p_{XY}^{\le r}(c,d)= p_{\be X}^{(r)}\be(c)\cdot p_{\le Y}^{(r)}\be(d\e)$, where $c\in H^{r}\be(X_{\et},\bg_{m,X}\be)$ and $d\in H^{r}\be(Y_{\et},\bg_{m,Y}\be)$. In particular, we obtain canonical maps
\begin{equation}\label{pxy0}
p_{\lbe XY}^{\le 0}\colon \bg_{m,\le S}\le(X\le)\oplus \bg_{m,\le S}\le(\le Y\le)\to\bg_{m,\le S}\le(X\!\times_{\lbe S}\! Y\le)
\end{equation}
and
\begin{equation}\label{pxy1}
p_{\lbe XY}^{\le 1}\colon \pic X\be\oplus\be\pic Y\to \pic\be(X\!\times_{\lbe S}\! Y\e).
\end{equation}
If $S$ is the spectrum of a field, the map \eqref{pxy0} (respectively, \eqref{pxy1}) was discussed previously by Jaffe \cite{jaf} and Conrad \cite{con} (respectively, Ischebeck \cite{isch} and Sansuc \cite{san}).

Now, since
\begin{equation}\label{fgp}
f\be\circ\be p_{\lbe X}=g\be\circ\be p_{\e Y}=f\be\times_{\be S}\be g,
\end{equation}
\eqref{need} shows that the following diagram commutes: 
\begin{equation}\label{+di}
\xymatrix{H^{r}\be(S_{\et},\bg_{m,S})\lbe\oplus\lbe H^{r}\be(S_{\et},\bg_{m,S})\,\ar@{->>}[d]^(.45){(\e\cdot\e)}\ar[drr]^{\theta^{\le r}_{\be XY}}\ar[rr]^(.48){\left( f^{(\lbe r\lbe)}\lbe,\, g^{(\lbe r\lbe)}\right)}&& \,H^{r}\be(X_{\et},\bg_{m,X}\be)\lbe\oplus\lbe H^{r}\be(\le Y_{\et},\bg_{m,Y}\be)\ar[d]^{\, p_{XY}^{\le r}}\\
H^{r}\be(S_{\et},\bg_{m,S})\ar[rr]_(.48){(\le f\lbe\times_{\be S}\le g\le)^{(r)}}&& H^{r}\be((X\!\times_{\be S}\! Y\le)_{\et},\bg_{m,\e X\lbe\times_{\be S}\lbe Y}\lbe),
}
\end{equation}
where the left-hand vertical map is the (surjective) homomorphism $(a,b)\mapsto ab$, the right-hand vertical map is the map \eqref{pxy} and the diagonal map above
\begin{equation}\label{txy}
\theta^{\le r}_{\be XY}\colon H^{r}\be(S_{\et},\bg_{m,S})\lbe\oplus\lbe H^{r}\be(S_{\et},\bg_{m,S}) \to H^{r}\be((X\!\times_{\be S}\! Y\le)_{\et},\bg_{m,\e X\lbe\times_{\be S}\lbe Y}\lbe)
\end{equation}
is the common composition $p_{XY}^{\le r}\be\circ\be \left( f^{(\lbe r\lbe)}\!,\le g^{(\lbe r\lbe)}\right)=(\le f\!\times_{\be S}\be g\le)^{(r)}\be\circ\be (\e\cdot\e)$.
The kernel of the map $(\e\cdot\e)$ in the above diagram can be identified with $H^{r}\be(S_{\et},\bg_{m,S})$ via the map $\nabla^{r}\!$ \eqref{delr}. Consequently, the left-hand vertical map in \eqref{+di} induces a map (also denoted $(\e\cdot\e)$)
\begin{equation}\label{a1}
(\e\cdot\e)\colon \krn\le f^{\le(r)}\be\oplus\be \krn\le g^{(r)}\to \krn\lbe (\le f\!\times_{\be S}\be g\le)^{(r)}
\end{equation}
so that the sequence
\begin{equation}\label{kks}
\begin{array}{rcl}
0\to\krn\lbe f^{\le(r)}\e\cap\, \krn g^{(r)}&\to&\krn f^{\le(r)}\e\oplus\e\e \krn g^{(r)}\overset{\!(\e\cdot\e)}{\lra}\krn\lbe (\le f\!\times_{\be S}\be g\le)^{(r)}\\\\
&\to& \krn\lbe(\le f\!\times_{\be S}\be g\le)^{(r)}\!\big/\!\be\left(\krn\lbe f^{\le(r)}\!\cdot\be\krn\lbe g^{(r)}\right)\to 0
\end{array}
\end{equation}
is exact, where the first nontrivial map is induced by \eqref{delr}. In particular, there exist canonical multiplication maps
\begin{equation}\label{a2}
(\e\cdot\e)\colon \krn\le \pic\be f\be\oplus\be \krn\le \pic \lbe g\to \krn\lbe \pic\be(\le f\!\times_{\be S}\be g\le)
\end{equation}
and
\begin{equation}\label{a3}
(\e\cdot\e)\colon \bxs\be\oplus\be \bys\to \brp(X\!\times_{\be S}\! Y/S\le).
\end{equation}

Note also that, by the commutativity of diagram \eqref{+di}, the right-hand vertical map in \eqref{+di} induces a map
\begin{equation}\label{pxyb}
\overline{p}_{XY}^{\e r}\colon \cok\be f^{\le(r)}\be\oplus\be\cok\be g^{(r)}\to \cok\be (\le f\!\times_{\be S}\be g\le)^{(r)}.
\end{equation}
In particular, there exists a canonical map
\begin{equation}\label{pxy1b}
\overline{p}_{XY}^{\e 1}\colon \npic\be(X\be/\be S\le)\oplus\npic\be(Y\be/\be S\le)\to\npic\be(X\!\times_{\be S}\be Y\be/\be S\le),
\end{equation}
where $\npic\be(-/\be S\le)$ is the functor \eqref{npicf}.

\begin{proposition}\label{kad} Let $f\colon X\to S$ and $g\colon Y\to S$ be morphisms of schemes such that $I\lbe(\le f\!\times_{\be S}\! g\le)$ is defined and is equal to $1$. Then, for every integer $r\geq 0$, the canonical multiplication map \eqref{a1} 
\begin{equation}\label{ned}
(\e\cdot\e)\colon \krn\le f^{\le(r)}\be\oplus\be \krn\le g^{(r)}\to \krn\lbe (\le f\!\times_{\be S}\be g\le)^{(r)}
\end{equation}
is an isomorphism, i.e.,
\[
\krn\lbe f^{\le(r)}\be\cap\be \krn g^{(r)}=\krn\lbe(\le f\!\times_{\be S}\be g\le)^{(r)}\!\big/\!\be\left(\krn\lbe f^{\le(r)}\!\cdot\be\krn\lbe g^{(r)}\right)=0
\]
(see \eqref{kks}). In particular, there exist canonical isomorphisms
\[
\krn \pic\be(\le f\!\times_{\be S}\be g\le)\simeq \krn\le \pic\be f\be\oplus\be \krn\le \pic \lbe g
\]
and
\[
\brp(X\!\times_{\be S}\! Y/S\le)\simeq \bxs\be\oplus\be \bys.
\]
\end{proposition}
\begin{proof} If $T\to S$ is a finite, \'etale and surjective morphism of constant degree $d$, then, by Remark \ref{sect}(a), the vertical compositions in the following diagram are the multiplication by $d$ maps:
\[
\xymatrix{\krn\le f^{\le(r)}\be\oplus\be \krn\le g^{(r)}\ar[r]\ar[d]^(.45){\!(\res,\res)}& \krn\lbe (\le f\!\times_{\be S}\be g\le)^{(r)}\ar[d]^(.45){\res}\\
\krn\le f_{T}^{\le(r)}\be\oplus\be \krn\le g_{T}^{(r)}\ar[r]\ar[d]^{(\tr,\tr)}&\krn\lbe (\le f_{T}\!\times_{\be T}\be g_{T}\le)^{(r)}\ar[d]^{\tr}\\
\krn\le f^{\le(r)}\be\oplus\be \krn\le g^{(r)}\ar[r]& \krn\lbe (\le f\!\times_{\be S}\be g\le)^{(r)}
}
\]
Now, if $T\to S$ is a quasi-section of $f\times_{S}g$, then $f_{T}\lbe\times_{\be T}\lbe g_{\le T}$ has a section and therefore $\krn (\le f_{T}\!\times_{\be T}\lbe g_{\le T})^{\le(r)}=\krn\lbe f_{T}^{\le(r)}=\krn\lbe g_{\le T}^{(r)}=0$ by Remark \ref{sect}(c). It follows from the above diagram that the kernel and cokernel of \eqref{ned} are annihilated by $d$. Since $I\lbe(\le f\!\times_{\be S}\! g\le)=1$, the proposition follows.
\end{proof}

Next we apply Lemma \ref{ker-cok} to the pair of maps
\[
H^{\le r}\be(S_{\et},\bg_{m,S})\lbe\oplus\lbe H^{r}\be(S_{\et},\bg_{m,S})\overset{\!(\e\cdot\e)}{\twoheadrightarrow}H^{r}\be(S_{\et},\bg_{m,S})\overset{(\le f\lbe\times_{\be S}\le g\le)^{(r)}}{\lra}H^{r}\be((X\!\times_{\be S}\! Y\le)_{\et},\bg_{m,\e X\lbe\times_{\be S}\lbe Y}\lbe),
\]
whose composition is $\theta^{\e r}_{\be XY}$ \eqref{txy}. We obtain a canonical isomorphism
\begin{equation}\label{iso1}
\cok\e\theta^{\e r}_{\be XY}\overset{\!\sim}{\to}\cok\lbe (\le f\!\times_{\be S}\be g\le)^{(r)}
\end{equation}
and a canonical exact sequence
\begin{equation}\label{sec1}
0\to H^{\le r}\be(S_{\et},\bg_{m,S})\overset{\!\nabla^{r}}{\lra}\krn\e \theta^{\e r}_{\be XY}\overset{\!(\e\cdot\e)}{\lra} \krn\lbe (\le f\!\times_{\be S}\be g\le)^{(r)}\to 0,
\end{equation}
where the first map is (induced by) the map \eqref{delr}.

Now we apply Lemma \ref{ker-cok} to the pair of maps
\[
\begin{array}{rcl}
H^{r}\be(S_{\et},\bg_{m,S})\oplus H^{r}\be(S_{\et},\bg_{m,S})&\overset{\!\left(f^{(\lbe r\lbe)}\!,\e g^{(\lbe r\lbe)}\lbe\right)}{\lra}& H^{r}\be(X_{\et},\bg_{m,X}\be)\oplus H^{r}\be(\le Y_{\et},\bg_{m,Y}\be)\\
&\overset{p_{\be XY}^{\le r}}{\lra}& H^{r}\be((X\!\times_{\be S}\lbe Y\le)_{\et},\bg_{m,\e X\lbe\times_{\be S}\lbe Y}\lbe),
\end{array}
\]
whose composition is also $\theta^{\e r}_{\be XY}$ \eqref{txy}, and obtain an exact sequence
\begin{equation}
\begin{array}{rcl}\label{6}
0&\to&\krn\le f^{\le(r)}\be\oplus\be \krn\le g^{(r)}\to \krn\e \theta^{\e r}_{\be XY}\to \krn\e p_{\be XY}^{\le r}\\
&\to&\cok\be f^{\le(r)}\be\oplus\be\cok\be g^{(r)}\to\cok \theta^{\e r}_{\be XY}\to\cok p_{\be XY}^{\le r}\to 0.
\end{array}
\end{equation}
Now observe that the map $\overline{p}_{\lbe XY}^{\e r}$ \eqref{pxyb} factors as
\[
\cok\be f^{\le(r)}\be\oplus\be\cok\be g^{(r)}\to\cok \theta^{\e r}_{\be XY}
\overset{\!\sim}{\to}\cok\lbe (\le f\!\times_{\be S}\be g\le)^{(r)},
\]
where the first (respectively, second) map is that appearing in \eqref{6} (respectively,  \eqref{iso1}). Thus we obtain from \eqref{6} a canonical exact sequence
\begin{equation}\label{imp}
0\to\krn\le f^{\le(r)}\be\oplus\be \krn\le g^{(r)}\to \krn\e \theta^{\e r}_{\be XY}\to \krn \e p_{\be XY}^{\le r}\to\krn\e\overline{p}_{\lbe XY}^{\e r}\to 0
\end{equation}
and a  canonical isomorphism
\[
\cok p_{\be XY}^{\le r}\simeq \cok \overline{p}_{\lbe XY}^{\e r}.
\]
Thus the following holds

\begin{proposition}\label{eqcor} Let $f\colon X\to S$ and $g\colon Y\to S$ be morphisms of schemes and let $r\geq 0$ be an integer. Then the cokernels of the maps 
\[
p_{\lbe XY}^{\le r}\colon H^{r}\be(X_{\et},\bg_{m,X}\be)\lbe\oplus\lbe H^{r}\be(\le Y_{\et},\bg_{m,Y}\be)\to H^{r}\!\left((X\!\times_{\be S}\! Y\le)_{\et},\bg_{m,\e X\lbe\times_{\be S}\lbe Y}\lbe\right)
\]
\eqref{pxy} and
\[
\overline{p}_{\lbe XY}^{\e r}\colon \cok\be f^{\le(r)}\be\oplus\be\cok\be g^{(r)}\to \cok\be (\le f\!\times_{\be S}\be g\le)^{(r)}
\]
\eqref{pxyb} are canonically isomorphic. In particular, the cokernels of the maps 
\[
\pic X\be\oplus\be\pic Y\!\to\!\pic\be(X\!\times_{\be S}\! Y\e)
\]
\eqref{pxy1} and
\[
\npic\be(X\be/\be S\le)\oplus\npic\be(Y\be/\be S\le)\!\to\!\npic\be(X\!\times_{\be S}\! Y\be/\be S\le)
\]
\eqref{pxy1b} are canonically isomorphic.\qed
\end{proposition}

In order to obtain information on $\krn\e p_{\lbe XY}^{\le r}$, we proceed as follows.

Consider the exact and commutative diagram
\begin{equation}\label{twin}
\xymatrix{& H^{r}\be(S_{\et},\bg_{m,S})\ar@{^{(}->}[d]^{\nabla^{r}}\ar@{-->}[dr]^(.45){\beta
^{\lle r}}&\\
\krn\le f^{\le(r)}\be\oplus\be \krn\le g^{(r)}\,\ar@{^{(}->}[r]\ar[dr]_(.45){(\e\cdot\e)}&\krn\e \theta^{\le r}_{\be XY}\ar@{->>}[d]^(.45){(\e\cdot\e)}\ar[r]^(.43){\!\left(\lbe f^{(\lbe r\lbe)}\!,\e g^{(\lbe r\lbe)}\be\right)}& \krn\e p_{\lbe XY}^{\e r}\ar@{->>}[r]&\krn\e\overline{p}_{\lbe XY}^{\e r}\\ 
&\krn\lbe (\le f\!\times_{\be S}\be g\le)^{(r)},&}
\end{equation}
where the row (respectively, column) is the sequence \eqref{imp} (respectively, \eqref{sec1}) and
\begin{equation}\label{beta}
\beta^{\le r}\overset{{\rm def.}}{=}\left(f^{(\lbe r\lbe)}\!,\e g^{(\lbe r\lbe)}\right)\be\circ\be \nabla^{r}\colon 
H^{r}\be(S_{\et},\bg_{m,S})\to \krn\e p_{\lbe XY}^{\le r}
\end{equation}
is the map $\xi\mapsto (f^{(r)}(\xi),g^{(r)}(\xi)^{-1}\le)$. Applying Lemma \ref{ker-cok} to the top triangle in \eqref{twin} and using the remaining parts of that diagram, we obtain an exact sequence
\[
\begin{array}{rcl}
0\to\krn\e\beta^{\le r}\to \krn f^{\le(r)}\e\oplus\e\e \krn g^{(r)}\overset{\!(\e\cdot\e)}{\lra}\krn\lbe (\le f\!\times_{\be S}\! g\le)^{(r)}&\to& \cok\e\beta^{\le r}\\
&\to&\krn\e\overline{p}_{\lbe XY}^{\e r}\to 0.
\end{array}
\]
Thus, by the exactness of \eqref{kks}, there exist a canonical isomorphism
\[
\krn\e\beta^{\le r}=\krn f^{\le(r)}\be\cap\be\krn g^{(r)}
\]
(this is also evident from the definition of $\beta^{\le r}\!$ \eqref{beta}) and a canonical exact sequence
\[
0\to\frac{\krn\lbe(\le f\!\times_{\be S}\be g\le)^{(r)}}{\krn\lbe f^{\le(r)}\!\cdot\be\krn\lbe g^{(r)}}\to\cok\lbe\beta^{\le r}\to \krn\e\overline{p}_{\lbe XY}^{\e r}\to 0.
\]
Thus the following holds.

\begin{proposition}\label{mad} Let $f\colon X\!\to\! S$ and $g\colon Y\!\to\! S$ be morphisms of schemes and let $\beta^{\le r}$ be the map \eqref{beta} associated to $f$ and $g$. Then, for every integer $r\geq 0$, there exists a canonical complex of abelian groups
\[
0\to\krn f^{\le(r)}\be\cap\be\krn g^{(r)}\to H^{\le r}\be(S_{\et},\bg_{m,S})\overset{\!\beta^{\le r}}{\lra} \krn p_{\lbe XY}^{\e r}\to\krn\e\overline{p}_{\lbe XY}^{\e r}\to 0
\]
which is exact except perhaps at $\krn p_{XY}^{\e r}$, where its homology is the quotient group $\krn\lbe(\le f\!\times_{\be S}\be g\le)^{(r)}\!\big/\!\be\left(\krn\lbe f^{\le(r)}\!\cdot\be\krn\lbe g^{(r)}\right)$.\qed
\end{proposition}

\begin{corollary}\label{cop} Let $f\colon X\to S$ and $g\colon Y\to S$ be morphisms of schemes such that $I\lbe(\le f\!\times_{\be S}\! g\le)$ is defined and is equal to $1$. Then, for every integer $r\geq 0$, there exists a canonical exact sequence of abelian groups
\[
0\to H^{\le r}\be(S_{\et},\bg_{m,S})\overset{\!\beta^{\le r}}{\lra} \krn p_{\lbe XY}^{\e r}\to\krn\e\overline{p}_{\lbe XY}^{\e r}\to 0,
\]
where $\beta^{\le r}$ is the map \eqref{beta} and the second nontrivial map is the canonical projection.
\end{corollary}
\begin{proof} This is immediate from the proposition using Proposition \ref{kad}.
\end{proof}

\begin{corollary}\label{cop0} Let $f\colon X\!\to\! S$ and $g\colon Y\!\to\! S$ be morphisms of schemes such that $I\lbe(\le f\!\times_{\be S}\! g\le)$ is defined and is equal to $1$. Assume, furthermore, that the canonical map $\npic\be(X\be/\be S\le)\oplus\npic\be(Y\be/\be S\le)\to\npic\be(X\!\times_{\be S}\be Y\be/\be S\le)$ \eqref{pxy1b} is injective. Then there exists a canonical exact sequence of abelian groups
\[
0\to\pic S\to\pic X\be\oplus\be\pic Y\to \pic\be(X\!\times_{\be S}\! Y\e).
\]
\end{corollary}
\begin{proof} Since $\overline{p}_{\lbe XY}^{\e 1}$ \eqref{pxy1b} is injective by hypothesis, the previous corollary shows that the map $\beta^{\le 1}\colon \pic S\to\krn\e p_{\lbe XY}^{\le 1}=\krn\!\lbe\left[\e\pic X\be\oplus\be\pic Y\to \pic\be(X\!\times_{\be S}\! Y\e)\e\right]$ \eqref{beta} is an isomorphism, where $p_{\lbe XY}^{\le 1}$ is the map \eqref{pxy1}. The corollary follows.	
\end{proof}

\section{The group of units of a product}\label{four}

The key to describing the kernel and cokernel of the map \eqref{pxy0} is the generalized
Rosenlicht additivity theorem of Corollary \ref{ros0} below. Recall that Rosenlicht's classical additivity theorem is the following statement: if $k$ is a field and $X$ and $Y$ are geometrically integral $k$-schemes of finite type, then the canonical homomorphism of abelian groups $U_{k}(X\le)\le\oplus\le U_{k}(\le Y\le)\to U_{k}(X\be\times_{k}\lbe\lbe Y\le)$ is an isomorphism, where $U_{k}(-)$ is the functor \eqref{uk} (see \cite[Theorem 2]{ros}). The preceding statement has been generalized by Conrad \cite{con} to the class of geometrically connected and geometrically reduced $k$-schemes locally of finite type. In this Section we extend Conrad's theorem to a relative setting, where the base $S=\spec k$ is replaced by an arbitrary reduced scheme. As in the case of a base field \cite{con}, the key result is Theorem \ref{ros1} below.

\smallskip

Let $f\colon X\to S$ and $g\colon Y\to S$ be morphisms of schemes and define a
morphism of \'etale sheaves on $S$
\begin{equation}\label{vt}
\vartheta\colon f_{\lbe *}\bg_{m,\le X}\oplus g_{*}\le\bg_{m,\le Y}\to (f\be\times_{\lbe S}\be g)_{*}\bg_{m,\e X\times_{\be S} Y}
\end{equation}
as follows:  via the isomorphisms \eqref{cb}, an element $(v,w)\in \bg_{m,\le S}(X\le)\oplus\bg_{m,\le S}(\le Y\le)$ may be regarded as a pair of $S$-morphisms $v\colon X\to \bg_{m,\e S}$ and $w\colon  Y\to \bg_{m,\e S}$. Then $\vartheta(S\le)$ is defined by mapping $(v,w)$ to the element of $\bg_{m,\e S}(X\be\times_{\be S}\lbe Y\le)$ which corresponds to the $S$-morphism
\begin{equation}\label{esm}
p_{XY}^{\le 0}\lbe(v,w)=p_{X}^{(0)}\be(v)\be\cdot\be p_{\e Y}^{(0)}\lbe(w)=(\e v\be\circ\be p_{X})\be\cdot\be(\e w\be\circ\be p_{\e Y})\colon X\times_{S}Y\to \bg_{m,\le S},
\end{equation}
where $p_{\lbe X}\colon X\!\times_{\be S}\! Y\!\to\! X$ and $p_{\e Y}\colon X\!\times_{\be S}\! Y\!\to\! Y$ are the canonical projections (see \eqref{iden} and \eqref{fo}). If $T\to S$ is any \'etale morphism, $\vartheta(T\le)\colon \bg_{m,\le T}\lle(X_{T})\oplus\bg_{m,\le T}\lle(\le Y_{T})\to \bg_{m,\le T}(X_{T}\!\times_{T}\! Y_{T})$ is defined similarly.

For every integer $r\geq 0$ we will make the identification
\[
H^{r}\be(S_{\et}, f_{\lbe *}\bg_{m,\le X}\oplus g_{*}\le\bg_{m,\le Y})=H^{\le r}\be(S_{\et},f_{\lbe *}\bg_{m,\le X})\oplus H^{\le r}\be(S_{\et},g_{\lbe *}\bg_{m,\le Y}).
\]
Now let
\[
\vartheta^{\le r}\colon H^{\le r}\be(S_{\et},f_{\lbe *}\bg_{m,\le X})\oplus H^{\le r}\be(S_{\et},g_{\lbe *}\bg_{m,\le Y})\to H^{\le r}\be(S_{\et},(\le f\!\times_{\be S}\be g)_{*}\bg_{m,\e X\times_{\be S} Y})
\]
be the map $H^{r}(S_{\et},\vartheta\le)$, i.e., 
\begin{equation}\label{vtr}
\vartheta^{\le r}\be(\xi,\zeta)=H^{\le r}\!\be\left(\le f_{\lbe *}\!\be\left(\le p_{\lbe X}^{\le\flat}\right)\right)\!(\xi\le)\!\cdot\! H^{\le r}\!\be\left(\le g_{\le *}\!\be\left(\le p_{\le Y}^{\le\flat}\right)\right)\!(\zeta\le)
\end{equation}
for all $\xi\in H^{\le r}(S_{\et},f_{\lbe *}\bg_{m,\le X})$ and $\zeta\in H^{\le r}(S_{\et},g_{\lbe *}\bg_{m,\le Y})$. When $r=0$, $\vartheta^{\e 0}$ agrees with $p_{XY}^{\le 0}$ (see \eqref{esm}):
\begin{equation}\label{t0}
\vartheta^{\e 0}\be(v,w)=H^{\le 0}\!\be\left(\le f_{\lbe *}\!\be\left(\le p_{\lbe X}^{\le\flat}\right)\right)\!(v)\!\cdot\! H^{\le 0}\!\be\left(\le g_{\le *}\!\be\left(\le p_{\le Y}^{\le\flat}\right)\right)\!(w)= p_{X}^{(0)}\be(v)\be\cdot\be p_{\e Y}^{(0)}\lbe(w)=p_{XY}^{\le 0}\be(v,w).
\end{equation}
By \eqref{need} and \eqref{fgp}, the following holds: for every $c\in\bg_{m,\le S}(\le S\le)$ and $(v,w)\in  \bg_{m,\le S}(X\le)\le\oplus\le\bg_{m,\le S}(\le Y\le)$, we have
\begin{equation}\label{vt3}
\vartheta^{\e 0}\lbe(\e f^{\le(0)}\lbe(c)\le v,w)=\vartheta^{\e 0}\lbe(\e v,g^{\le(0)}\lbe(c)\le w).
\end{equation}

\begin{remark}\label{warn} Since, in general, the equality $H^{\le r}\be(S_{\et},f_{\lbe *}\bg_{m,\le X})=H^{\le r}\be(X_{\et},\bg_{m,\le X})$ holds only for $r=0$, the developments of this Section do not lead to an additivity theorem for the Picard group of $X\be\times_{\be S}\be Y$. They do, however, yield one for its subgroup \eqref{sbgp}. See Corollary \ref{upb} below.
\end{remark}

Now recall the \'etale sheaves $\uxs, \uys$ and $U_{X\times_{\be S}Y/S}$ \eqref{uxs2}.

\begin{theorem}\label{ros1} Let $S$ be a reduced scheme and let $f\colon X\to S$ and $g\colon Y\to S$ be faithfully flat morphisms locally of finite presentation. Assume that 
\begin{enumerate}
\item[(i)] the maximal geometric fibers of $f$ and $g$ are reduced and connected and
\item[(ii)] $f\!\times_{S}\be g\e\colon X\be\times_{S}\be Y\to S$ has a section.
\end{enumerate}
Then, for every \'etale morphism $T\to S$, the canonical homomorphism of abelian groups
\[
\uxs\lbe(\le T\e)\lbe\oplus\lbe\uys(\le T\e)\to U_{X\times_{\be S}Y/S}\le(\le T\e)
\]
is an isomorphism.
\end{theorem}
\begin{proof} We begin by observing that, since $S$ is reduced and $T\to S$ is \'etale, $T$ is reduced as well by \cite[${\rm IV}_{4}$, Proposition 17.5.7]{ega}. On the other hand, Lemma \ref{mf} shows that $f_{T}$ and $g_{\e T}$ satisfy the remaining conditions of the theorem  (over $T$). Thus, by Lemma \ref{sd-ci}, we may henceforth assume that $T=S$. 
Now let $\rho\le\colon\be S\to X\be\times_{\lbe S}\be Y$ denote the given section of $f\be\times_{S}\lbe g$. Then $\sigma=p_{X}\lbe\circ\lbe\rho\colon S\to X$ and $\tau=p_{\e Y}\be\circ\be\rho\colon S\to Y$ are sections of $f$ and $g$, respectively. Consequently, the sequence of \'etale sheaves on $S$
\[
1\to\bg_{m,\le S}\overset{\! h^{\lbe\flat}}{\to}h_{\lbe *}\bg_{m,\le Z}\to U_{Z/S}\to 1
\]
is split exact if $h\colon Z\to S$ is any one of $f,g$ or $f\be\times_{S} g$. Thus there exists a canonical commutative diagram whose rows are split exact sequences of abelian groups 
\begin{equation}\label{di}
\xymatrix{\bg_{m,\le S}\le(S\le)\be\oplus\be \bg_{m,\le S}\le(S\le)\,\ar@{^{(}->}[rr]^(.49){(\le f^{(0)}\be,\, g^{(0)})}\ar[d]^{(\e\cdot\e)}&& \bg_{m,\le S}(X\le)\oplus\bg_{m,\le S}(\le Y\le) \ar@{->>}[r]\ar[d]^(.45){\vartheta^{\lle 0}}& \uxss\le\oplus\le \uyss\ar[d]\\
\bg_{m,\le S}\le(S\le)\,\ar@{^{(}->}[rr]^(.5){(\e f\times_{\be S}\e g)^{(0)}}&& \bg_{m,\e S}(X\!\times_{\be S}\! Y\le)\ar@{->>}[r]& U_{\be S}(X\!\times_{\be S}\! Y\le),
}
\end{equation}
where the left-hand vertical arrow is the multiplication map and $U_{\be S}\colon ({\rm{Sch}}/S\e)\to \mathbf{Ab}$ is the functor \eqref{uxss}.

We first establish the injectivity of the right-hand vertical map in diagram \eqref{di}. Let $v\colon X\to \bg_{m,\e S}$ and $w\colon  Y\to \bg_{m,\e S}$ be $S$-morphisms such that the $S$-morphism $\vartheta^{\le 0}\lbe(v,w)=(v\le\circ\le p_{\lbe X}\lbe)\lbe\cdot\lbe (w\le \circ\le  p_{\e Y}\lbe)\colon X\times_{S}Y\to \bg_{m,\le S}$ lies in the image of the bottom left map $(\e f\be\times_{\be S}\be\e g)^{(0)}$ in diagram \eqref{di}, i.e,
\begin{equation}\label{clear}
(v\be\circ\be p_{\lbe X})\lbe\cdot\lbe (w\be \circ\be  p_{\e Y})=c\be\circ\be (\le f\be\times_{\be S}\be g\le)
\end{equation}
for some $S$-morphism $c\colon S\to \bg_{m,\le S}$ (see \eqref{iden} and \eqref{fo}). Since $\sigma_{Y}$ is a section of $p_{\e Y}$, we have\,\,\footnote{Here we use freely certain well-known identifications, e.g., $p_{X}=1_{X}\!\times_{\be S}\! g$,  $g=1_{S}\!\times_{\be S}\! g$, etc. See \cite[Corollary 1.2.8, p.~28]{ega1}.}
\[
\begin{array}{rcl}
(v\be\circ\be p_{\lbe X})\be\circ\be\sigma_{Y}&=& v\be\circ\be\sigma\be\circ\be  g=g^{\le(0)}(v\be\circ\be\sigma),\\
(w\be \circ\be p_{\e Y})\be \circ\be\sigma_{Y}&=& w,\\
c\be\circ\be(\le f\be \times_{\be S}\be g\le)\be \circ\be \sigma_{Y}&=& c\be\circ\be g=g^{\le(0)}(c).
\end{array}
\]
Composing both sides of the identity \eqref{clear} with $\sigma_{Y}$ from the right and using the preceding formulas, we obtain $g^{\le(0)}\lbe(v\lbe\circ\lbe\sigma)\cdot w=g^{\le(0)}(c)$, whence $w\in\img g^{\le(0)}$. A similar argument, using the section $\tau_{\lbe X}\colon X\to X\be\times_{S}\be Y$ of $p_{\lbe X}$ in place of $\sigma_{Y}$, shows that $v\in \img f^{\le(0)}$.

We now prove the surjectivity of the right-hand vertical map in diagram \eqref{di}. It suffices to show that the map $\vartheta^{\le 0}$ in \eqref{di} is surjective.  The proof combines the surjectivity of this map over a field, i.e., \cite[Theorem 1.1]{con}, and a density argument of Raynaud \cite[proof of Corollary VII 1.2, p.~103]{ray}. Let $u\colon X\times_{\lbe S} Y\to \bg_{m,\e S}$ be an $S$-morphism, i.e., an element of $\bg_{m,\e S}(X\!\times_{\be S}\be Y\le)$, and consider the following $S$-morphisms $X\times_{\lbe S} Y\to \bg_{m,\e S}\,$:
\[
\begin{array}{rcl}
u_{\e 00}&=&(\e f\be\times_{\be S}\be g)^{(0)}(\e \rho^{\le(0)}\lbe(\le u\le))=p_{ X}^{\le(0)}\lbe(\le f^{\le(0)}\lbe(\e \rho^{\le(0)}\lbe(\le u\le)))=p_{\e Y}^{\le(0)}\lbe(\e g^{\le(0)}(\e \rho^{\le(0)}\lbe(\le u\le))),\\
u_{\e 10}&=& p_{\be X}^{\le(0)}\lbe(\le\tau_{\lbe X}^{\le(0)}\lbe(\le u\le)),\\
u_{\e 01}&=& p_{\le Y}^{\le(0)}\lbe(\sigma_{Y}^{\le(0)}\lbe(\le u\le))
\end{array}
\]
(recall that $\rho\le\colon\be S\to X\be\times_{\lbe S}\be Y$ is the given section of $f\!\times_{\lbe S}\! g\e$). Then $u$ has the form $\vartheta^{\le 0}\be(v,w)=p_{\lbe X}^{\le(0)}\be(\le v)\lbe\cdot\lbe p_{\le  Y}^{\le(0)}\be(\le w)$ for some $S$-morphisms $v\colon X\to \bg_{m,\e S}$ and $w\colon  Y\to \bg_{m,\e S}$ if, and only if, the $S$-morphism
\begin{equation}\label{uti}
\widetilde{u}=u\be\cdot\be u_{\le 00}\be\cdot\be u_{10}^{-1}\be\cdot\be u_{01}^{-1}\,\colon\, X\be\times_{\lbe S}\be Y\to \bg_{m,\e S}
\end{equation}
factors through the unit section $\varepsilon$ of $\bg_{m,\e S}$. Let $\varepsilon(\le S\le)$ denote the closed subscheme of $\bg_{m,\e S}$ determined by $\varepsilon$,  let $W=\widetilde{u}^{-1}(\varepsilon(\le S\le))$ be the corresponding closed subscheme of $X\be\times_{S}\lbe Y$ \cite[\S 4.3, p.~263]{ega1} and write $M$ for the set of maximal points of $S$. For every $\eta\in M$, set $k=k(\eta)$ and let $(x_{0},y_{0})\in X\be(k)\!\times\! Y\be(k)$ correspond to $\rho_{\le\eta}\colon\eta\to(X\be\times_{\lbe S}\be Y\le)_{\eta}$. By hypothesis (i), \cite[Theorem 1.1]{con} holds, i.e., the sequence \eqref{c1} is exact. Thus every unit $u(x,y)$ on $X\times_{k}Y$ has the form 
$u_{\be X}\be(x)u_{Y}\be(y)$ for some units $u_{\be X}$ and $u_{Y}$ on $X$ and $Y$, respectively. Consequently, $u(x_{0},-)=u_{X}(x_{0})u_{Y}$, $u(-,y_{0})=u_{Y}(y_{0})u_{\be X}$ and $u_{\be X}\lbe(x_{0})u_{Y}\lbe(y_{0})=u(x_{0},y_{0})$, whence
\[
u(x,y)u(x_{0},y_{0})u(x,y_{0})^{-1}u(x_{0},y)^{-1}=1
\]
(the preceding argument is due to Conrad \cite[beginning of \S 2]{con}). The units $u(x_{0},y_{\le 0}),u(x,y_{0})$ and $u(x_{0},y)$ are instances of the units $u_{\e 00}, u_{\e 10}$ and $u_{\e 01}$ (respectively) when $S=\spec k$. It follows that the $\eta\e$-morphism $\widetilde{u}_{\e\eta}\colon (X\be\times_{\lbe S}\be Y\le)_{\eta}\to \bg_{m,\e\eta}$ \eqref{uti} factors through the unit section $\varepsilon_{\e\eta}$ of $\bg_{m,\e \eta}$ for every $\eta\in M$. We conclude that $W$ majorizes the maximal fibers $(X\be\times_{\lbe S}\be Y\le)_{\eta}$ of
$f\be\times_{\be S}\lbe g$ . On the other hand, since $S$ is reduced, the family of canonical morphisms $\{\eta\to S\}$ is schematically dominant \cite[${\rm IV}_{3}$, Proposition 11.10.4]{ega}, whence $\{(X\times_{\be S} Y\le)_{\eta}\to X\times_{\be S} Y\}$ is schematically dominant as well \cite[${\rm IV}_{3}$, Theorem 11.10.5(ii)(b)]{ega}. Thus $W=X\be\times_{S}\be Y$ by  \cite[Corollary 5.4.4, p.~285]{ega1}, which completes the proof. 
\end{proof}

\begin{corollary} \label{ros0} {\rm (Rosenlicht's additivity theorem over a reduced base)} Let $S$ be a reduced scheme and let $f\colon X\to S$ and $g\colon Y\to S$ be faithfully flat morphisms locally of finite presentation whose maximal geometric fibers are reduced and connected. If $f\!\times_{\lbe S}\be g\e\colon X\be\times_{S}\be Y\to S$ has an \'etale quasi-section, then the canonical morphism $\uxs\lbe\oplus\lbe\uys\to U_{X\times_{S}Y/S}$ \eqref{fxy} is an isomorphism of \'etale sheaves on $S$.
\end{corollary}
\begin{proof} By \cite[IV, Corollary 4.5.8]{sga3} and Lemma \ref{sd-ci}, we may replace $S,f$ and $g$ by $T, f_{T}$ and $g_{\e T}$ (respectively) for an appropriate \'etale and surjective morphism $T\to S$. Now, since $f\!\times_{\lbe S}\be g\e\colon X\be\times_{S}\be Y\to S$ has an \'etale quasi-section, there exists a morphism $T\to S$ as above such that $T,f_{T}$ and $g_{\e T}$ satisfy all the hypotheses of the theorem. The corollary follows.
\end{proof}

We will now develop a number of consequences of the preceding corollary.

\smallskip

Under the hypotheses of Corollary \ref{ros0}, the following holds. There exists an exact and commutative diagram of abelian sheaves on $S_{\et}$:
\begin{equation}\label{bbd}
\xymatrix{1\ar[r]&\bg_{m,\le S}\be\oplus\be \bg_{m,\le S}\,\ar[r]^(.42){(\e f^{\le\flat}\!,\, g^{\le\flat})}\ar@{->>}[d]^(.47){(\e\cdot\e)}& f_{\lbe *}\bg_{m,\le X}\oplus g_{*}\le\bg_{m,\le Y}\ar[r]\ar[d]^{\vartheta}& \uxs\le\oplus\le\uys\ar[d]^{\simeq}\ar[r]&1\\
1\ar[r]&\bg_{m,\le S}\,\ar[r]& (f\be\times_{\lbe S}\be g)_{*}\bg_{m,\e X\times_{\be S} Y}\ar[r]& U_{X\times_{S}Y/S}\ar[r]& 1,
}
\end{equation}
where the middle vertical morphism is the map \eqref{vt}. The kernel of $(\e\cdot\e)$ can be identified with $\bg_{m,\le S}$ via the morphism $\nabla$ \eqref{nb}. Thus the diagram yields an exact sequence of \'etale sheaves on $S$
\begin{equation}\label{cseq}
1\to \bg_{m,\le S}\overset{\delta}{\to} f_{\lbe *}\bg_{m,\le X}\!\oplus\be g_{\lle *}\le\bg_{m,\le Y}\overset{\!\be\vartheta}{\to}(f\!\times_{\be S}\be g)_{*}\bg_{m,\e X\times_{\be S} Y}\to 1,
\end{equation}
where $\delta=(\le f^{\le\flat}\be,g^{\le\flat})\!\circ\!\nabla$. Note that, for every integer $r\geq 0$, the map
\[
\delta^{\le r}=H^{r}\be(S_{\et},\delta)\colon H^{\le r}\be(S_{\et},\bg_{m,\le S})\to H^{\le r}\be(S_{\et},f_{\lbe *}\bg_{m,\le X})\oplus H^{\le r}\be(S_{\et},g_{\lbe *}\bg_{m,\le Y}\be)
\]
is given explicitly by
\begin{equation}\label{psir}
\delta^{\le r}=\left(\le H^{\le r}\be\!\left(\le f^{\e\flat}\le\right)\le,H^{\le r}\!\be\left(\le g^{\e\flat}\le\right)^{\be-1}\le\right).
\end{equation}
When $r=0$, we have
\begin{equation}\label{de0}
\delta^{\le 0}\lbe(c)=\left(\e f^{\le(0)}\lbe(c),g^{\le (0)}\lbe(c)^{-1}\right)
\end{equation}
for every $c\in\bg_{m,\le S}(\le S\le)$ (see \eqref{fo}). Note also that 
$\delta^{\le 0}=\beta^{\le 0}$ \eqref{beta} but that, in general, $\beta^{\le r}\neq \delta^{\le r}$ for $r>0$ (see Remark \ref{warn}).

\begin{corollary}\label{kres} Let $S$ be a reduced scheme and let $f\colon X\to S$ and $g\colon Y\to S$ be faithfully flat morphisms locally of finite presentation with reduced and connected maximal geometric fibers. Assume that
\begin{enumerate}
\item[(i)] $f$ has an \'etale quasi-section and 
\item[(ii)] $g$ has a section $\sigma\colon S\to Y$.
\end{enumerate}
Then $\sigma$ induces an isomorphism of \'etale sheaves on $S$
\[
f_{\lle *}\le\bg_{m,\le X}\oplus\uys\overset{\!\sim}{\to}
(f\!\times_{\be S}\! g)_{*}\bg_{m,\e X\times_{\be S} Y}.
\]
Further, \eqref{cseq} is a split exact sequence of abelian sheaves on $S_{\et}$.
\end{corollary}
\begin{proof} 
Let $\alpha\colon T\to S$ be an \'etale quasi-section of $f$, i.e.,
$\alpha$ is \'etale and surjective and there exists a morphism $h\colon T\to X$ such that $f\circ h=\alpha$. Then the morphism $\beta=\alpha\times_{S}1_{S}\colon T\times_{S}S\to S\times_{S}S\simeq S$ is \'etale and surjective and $h\times_{S}\sigma\colon T\times_{S}S\to X\times_{S}Y$ satisfies
$(\e f\be\times_{\be S}\lbe g)\lbe\circ\lbe(\e h\be\times_{\be S}\lbe \sigma)=\beta$, i.e., $\beta$ is an \'etale quasi-section of $f\be\times_{\be S}\lbe g$. Thus the conditions of Corollary \ref{ros0} are satisfied, whence diagram \eqref{bbd} is exact and commutative and the sequence \eqref{cseq} is exact. Now $\sigma$ defines a retraction of $g^{\le\flat}\colon \bg_{m,\le S}\to g_{*}\bg_{m,\le Y}$ which, in turn, defines a section $\tau\colon\uys\to g_{*}\bg_{m,\le Y}$ of the canonical morphism $g_{*}\bg_{m,\le Y}\to\uys$. Let
\[
\vartheta^{\,\prime}=\vartheta\circ (\e 1_{f_{*}\lbe\bg_{m,\le X}},\tau)\colon f_{*}\le\bg_{m,\le X}\oplus\uys \to (f\!\times_{\be S}\! g)_{*}\bg_{m,\e X\times_{\be S} Y}.
\]
The injectivity of $\vartheta^{\e\prime}$ follows from diagram \eqref{bbd}. Its surjectivity follows from the surjectivity of $\vartheta$, \eqref{vt3} and the formula $g_{*}\bg_{m,\le Y}=\img g^{\le\flat}\!\be\cdot\be\img \tau$. Finally, the map
$(\e 1_{f_{*}\lbe\bg_{m,\le X}},\tau)\circ(\vartheta^{\e\prime}\le)^{-1}
\colon (f\!\times_{\be S}\be g)_{*}\bg_{m,\e X\times_{\be S} Y}\to f_{\lbe *}\bg_{m,\le X}\!\oplus\be g_{\lle *}\le\bg_{m,\le Y}$ is a section of $\vartheta$ which splits \eqref{cseq}.
\end{proof}

\smallskip

Assume again that the hypotheses of Corollary \ref{ros0} hold, so that \eqref{cseq} is an exact sequence of \'etale sheaves on $S$. Then \eqref{cseq} induces an exact sequence of abelian groups
\[
\begin{array}{rcl}
\dots\to H^{\le r}(S_{\et},\bg_{m,\le S})&\overset{\!\delta^{\le r}}{\lra}&
H^{\le r}(S_{\et},f_{\lbe *}\bg_{m,\le X})\oplus H^{\le r}(S_{\et},g_{\lbe *}\bg_{m,\le Y})\\
&\overset{\!\vartheta^{\le r}}{\lra}& H^{\le r}(S_{\et},(\le f\!\times_{\be S}\be g)_{*}\bg_{m,\e X\times_{\be S} Y})\to H^{\le r+1}(S_{\et},\bg_{m,\le S})\overset{\!\delta^{\le r+1}}{\lra}\dots
\end{array}
\]
where, for every $r\geq 0$, the maps $\delta^{\lle r}$ and $\vartheta^{\le r}$ are given by \eqref{psir} and \eqref{vtr}, respectively. Consequently, for every integer $r\geq 0$, there exists a canonical exact sequence of abelian groups
\begin{equation}\label{krr}
\begin{array}{rcl}
0&\to &\krn H^{r}\be(\le f^{\le\flat})\cap \krn H^{r}\be(\le g^{\le\flat})\to H^{\le r}(S_{\et},\bg_{m,\le S})\overset{\!\delta^{\le r}}{\to} H^{\le r}\be(S_{\et},f_{\lbe *}\bg_{m,\le X})\oplus H^{\le r}\lbe(S_{\et},g_{\lle *}\bg_{m,\le Y})\\
&\overset{\!\vartheta^{\le r}}{\to}& H^{\le r}(S_{\et},(\le f\!\times_{\be S}\be g)_{*}\bg_{m,\e X\times_{\be S} Y})\to \krn H^{\le r+1}\be(\le f^{\le\flat})\cap \krn H^{\le r+1}\be(\le g^{\le\flat})\to 0,
\end{array}
\end{equation}
where the intersections take place inside $H^{\le r}\be(S_{\et},\bg_{m,\le S})$ and $H^{\le r+1}\lbe(S_{\et},\bg_{m,\le S})$, respectively. Note that, by \eqref{mag2}, 
\[
\krn H^{1}\be(\le f^{\e\flat})\cap \krn H^{1}\be(\le g^{\le\flat})=\krn\pic f\le\cap\le
\krn\pic g.
\]
Thus, setting $r=0$ in \eqref{krr}, we obtain the following statement, which solves the problem of describing the kernel and cokernel of the map \eqref{pxy0} under the hypotheses stated below:

\begin{corollary} \label{prec} Let $S$ be a reduced scheme and let $f\colon X\to S$ and $g\colon Y\to S$ be faithfully flat morphisms locally of finite presentation with reduced and connected maximal geometric fibers. If $f\!\times_{\lbe S}\be g\e\colon X\be\times_{\lbe S}\be Y\to S$ has an \'etale quasi-section, then \eqref{cseq} induces an exact sequence of abelian groups
\[
1\to \bg_{m,\le S}\le(S\le)\overset{\!\delta^{\le 0}}{\to} \bg_{m,\le S}\le(X\le)\oplus \bg_{m,\le S}\le(\le Y\le)\overset{\!\lbe\vartheta^{\le 0}}{\to}\bg_{m,\le S}\le(X\!\times_{\lbe S}\! Y\le)\to \krn\pic f\le\cap\le \krn\pic g\to 0,
\]
where the intersection takes place inside $\pic S$ and the maps $\delta^{\le 0}$ and  $\vartheta^{\le 0}$ are given by \eqref{de0} and \eqref{t0}, respectively.\qed
\end{corollary}

If $r\geq 0$ is arbitrary, the following holds.

\begin{corollary}\label{upb} Let $S$ be a reduced scheme and let $f\colon X\to S$ and $g\colon Y\to S$ be faithfully flat morphisms locally of finite presentation with reduced and connected maximal geometric fibers. Assume that one of the \'etale indices $I(\le f\le)$ or $I(\le g\le)$ is defined and is equal to $1$, or
both are defined and $\gcd(I(f\le), I(g\le))=1$. Then, for every integer $r\geq 0$, there exists a canonical exact sequence of abelian groups
\[
0\to H^{\le r}\be(S_{\et},\bg_{m,\le S})\overset{\!\delta^{\le r}}{\to} H^{ r}\be(S_{\et},f_{\lbe *}\bg_{m,\le X})\oplus H^{\le r}\be(S_{\et},g_{\lbe *}\bg_{m,\le Y})\overset{\!\vartheta^{\le r}}{\to} H^{\le r}(S_{\et},(\le f\times_{\be S} g)_{*}\bg_{m,\e X\times_{\be S} Y})\to 0\e,
\]
where the maps $\delta^{\le r}$ and $\vartheta^{\le r}$ are given by \eqref{psir} and \eqref{vtr}, respectively. In particular, there exists a canonical exact sequence of abelian groups
\[
0\to \pic\e S\to H^{1}\be(S_{\et},f_{\lbe *}\bg_{m,\le X})\oplus H^{\le 1}\be(S_{\et},g_{\lbe *}\bg_{m,\le Y})\to H^{\le 1}(S_{\et},(\le f\!\times_{\be S}\be g)_{*}\bg_{m,\e X\times_{\be S} Y})\to 0,\e
\]
where the groups $H^{1}\be(S_{\et},f_{\lbe *}\bg_{m,\le X}),  H^{\le 1}\be(S_{\et},g_{\lbe *}\bg_{m,\le Y})$ and $H^{\le 1}(S_{\et},(\le f\!\times_{\be S}\be g)_{*}\bg_{m,\e X\times_{\be S} Y})$ can be identified with subgroups of $\pic X, \pic Y$ and $\pic(X\!\times_{\be S}\be Y)$, respectively \eqref{sbgp}. If either $f$ or $g$ has a section, then both sequences above are split exact. 
\end{corollary}
\begin{proof} Recall the exact sequence \eqref{krr} and note that, for every $r\geq 0$, \eqref{mag} yields 
\[
\krn H^{r}\be(\le f^{\le\flat})\cap \krn H^{r}\be(\le g^{\le\flat})\subset\krn f^{\le(r)}\le\cap\le \krn g^{(r)},
\]
where the intersections take place inside $H^{\le r}\be(S_{\et},\bg_{m,\le S})$. Under either of the hypotheses on the \'etale indices of $f$ and $g$, $\krn f^{\le(r)}\le\cap\le \krn g^{(r)}$ vanishes for every $r\geq 0$ by Remark \ref{sect}(d). The first assertion of the corollary now follows from the exactness of \eqref{krr}. Further, if either $f$ or $g$ has a section then, by Corollary \ref{kres}, the sequence of \'etale sheaves \eqref{cseq} is split exact. The last assertion of the corollary is then clear.
\end{proof}

\begin{remarks}\label{ppic} Let $k$ be a field and let $X$ and $Y$ be $k$-schemes.
\begin{enumerate}
\item[(a)] If the separable indices of $X$ and $Y$ \eqref{ind} are defined and are coprime, then the corollary shows that, for every integer $r\geq 0$, there exists a canonical exact sequence of Galois cohomology groups 
\[
0\to H^{\le r}\be(k,\bg_{m,\le k}\lbe)\to H^{\le r}\be(k,\ks[X]^{*})\oplus H^{\le r}\be(k,\ks[Y]^{*})\to H^{\le r}\be(k,\ks[\le X\!\times_{k}\!Y\le]^{*})\to 0.
\]
In particular, by Hilbert's theorem 90, the canonical map
\[
H^{\le 1}\be(k,\ks[X]^{*})\oplus H^{\le 1}\be(k,\ks[Y]^{*})\overset{\!\sim}{\to} H^{\le 1}\be(k,\ks[\le X\!\times_{k}\!Y\le]^{*})
\]
is an isomorphism. Further, if either $X(k)\neq\emptyset$ or $Y(k)\neq\emptyset$, then the above sequence is split exact.
\item[(b)] In general, by Hilbert's Theorem 90 and the exactness of \eqref{krr} for $S=\spec k$ and $r=1$, the canonical map
\[
H^{\le 1}(k,\ks[X]^{*})\oplus H^{\le 1}(k,\ks[Y]^{*})\to H^{\le 1}(k,\ks[\le X\!\times_{k}\!Y\le]^{*})
\]
is injective.
\end{enumerate}
\end{remarks}

\smallskip

We conclude this Section with an additional application of Corollary \ref{ros0}, namely Proposition \ref{tors} below.

\smallskip

If $S$ is a scheme and $G$ is an $S$-group scheme, the \'etale {\it sheaf of characters of $G$} is the \'etale sheaf  $G^{\e *}$ on $S$ defined thus: if $T\to S$ is an \'etale morphism, then
\[
G^{\e *}\be(T\le)=\Hom_{\e T{\text -{\rm gr}}}(G_{\le T},\bg_{m,\e T}).
\]

\begin{lemma}\label{gros} Let $S$ be a reduced scheme and let $G$ be a 
flat $S$-group scheme locally of finite presentation with smooth and connected maximal fibers. Then there exists a canonical isomorphism of \'etale sheaves on $S$ \[
U_{G\lbe/\lbe S}=G^{\e *}.
\] 
\end{lemma}
\begin{proof} Let $T\to S$ be any \'etale morphism. By Lemma \ref{sd-ci} and the equality $G^{\e *}\be(T\le)=(G_{T}\lbe)^{*}(T\le)$, for the purpose of defining a canonical isomorphism of groups $U_{G/S}(T\e)\overset{\!\sim}{\to} G^{\le *}(T\e)$ we may assume that $T=S$. Let $f\colon G\to S$ be the structure morphism of $G$ and let $\varepsilon\colon S\to G$ be its unit section. For every $S$-morphism $v\colon G\to \bg_{m,\le S}$, consider the $S$-morphism $\widetilde{v}=v\lbe\cdot\lbe(v\lbe\circ\lbe\varepsilon\lbe\circ\lbe f\le)^{-1}=
v\lbe\cdot\lbe (f^{\le(0)}\lbe(\varepsilon^{\le (0)}\lbe(v)))^{-1}\colon G\to \bg_{m,\le S}$. Since $\widetilde{v}\be\circ\be\varepsilon\colon S\to \bg_{m,\le S}$ factors through the unit section of $\bg_{m,\le S}$, $\widetilde{v}\in G^{\le *}\be(S\le)$ by \cite[VII, Corollary 1.2, p.~103]{ray}. Thus we obtain a map $(\e f_{\lbe *}\bg_{m,\e G})(S\le)\to G^{\le *}\be(S\le), \e v\mapsto \widetilde{v}$, whose kernel is easily seen to be equal to the image of $f^{\le(0)}\colon\bg_{m,\e S}(S\le)\to (\e f_{\lbe *}\bg_{m,\e G})(S\le)$. Since $\widetilde{\chi}=\chi$ for every $S$-homomorphism $\chi\colon G\to \bg_{m,\le S}$, i.e., $\chi\in G^{\le *}\be(S\le)$, the indicated map is also surjective, whence the lemma follows.
\end{proof}

\begin{proposition}\label{tors} Let $S$ be a reduced scheme, $G$ a 
flat $S$-group scheme locally of finite presentation with smooth and connected maximal fibers and $f\colon X\to S$ a torsor under $G$ over $S$. If $f$ has an \'etale quasi-section, then there exists a canonical isomorphism of \'etale sheaves on $S$
\[
\uxs=G^{\e *}.
\]
\end{proposition}
\begin{proof} Let $g\colon G\to S$ be the structural morphism of $G$ and  $\varepsilon \colon S\to G$ its unit section. If $\alpha\colon T\to S$ is an \'etale quasi-section of $f$ and $h\colon T\to X$ is the corresponding $S$-morphism, then $(h,\varepsilon\be\circ\be \alpha\e)_{S}\colon T\to X\be\times_{S}\lbe G$ and $(h,h)_{S}\colon T\to X\be\times_{S}\be X$ are $S$-morphisms, whence $\alpha$ is also a quasi-section of both $f\lbe\times_{\lbe S}\e g$ and $f\lbe\times_{S}\lbe f$. Further, the $S$-isomorphism \eqref{ro} induces an isomorphism of $X_{T}\e$-schemes $\rho_{\e T}\colon X_{T}\be \times_{T}\be G_{\le T}\overset{\!\sim}{\to} X_{T}\be \times_{T}\be X_{T}$. Regarding $T$ as an $X_{T}\e$-scheme via the section $(h,1_{T}\e)_{S}\colon T\to X_{T}$ of $f_{T}$, we obtain a $T$-isomorphism $1_{T}\times_{\be X_{T}}\be\rho_{\e T}\colon G_{T}\overset{\!\sim}{\to}X_{T}$. It now follows from Lemma \ref{mf} that the pairs of morphisms $(\e f,g\e)$ and $(\e f,f\e)$ satisfy all the conditions of Corollary \ref{ros0}. Therefore the canonical morphisms $\uxs\oplus\ugs\to U_{X\times_{S}\e G/S}$ and $\uxs\oplus\uxs\to U_{X\times_{S}X/S}$ are isomorphisms of \'etale sheaves on $S$. The proposition now follows by combining Lemmas \ref{gros} and \ref{add}.
\end{proof}

\section{The Picard group of a product}\label{5}

In this Section we establish, under certain conditions, an additivity theorem for the Picard group of a product of schemes over a normal base. See Theorem \ref{non2}. All schemes below are assumed to be locally noetherian.

\smallskip

We begin with 

\begin{proposition}\label{inj} Let $X$ and $Y$ be schemes over a field $k$ with separable closure $\ks$ such that $(X\times_{k} Y\le)(\ks)\neq\emptyset$. Then the canonical map $\pic X\be\oplus\lbe \pic Y\to \pic\be(X\!\times_{\lbe k}\! Y\le)$ \eqref{pxy1} is injective.
\end{proposition}
\begin{proof} This can be proved using the Galois descent argument given in \cite[proof of Lemma 11, pp.~188-189]{cts77} using Remark \ref{ppic}(b).
\end{proof}

If $f\colon X\to S$ is a morphism of schemes and $\eta$ is a maximal point of $S$, we will write  $j_{\le\eta}\colon X_{\eta}\to X$ for the base extension along $f$ of the canonical embedding $\eta\to S$. There exists a canonical homomorphism of abelian groups
\begin{equation}\label{pico}
(\le\pic\be\le j_{\le\eta})\colon \pic X\to\prod\pic X_{\eta}\e,\e c\mapsto((\pic\lbe j_{\le\eta})(c))_{\eta}\,,
\end{equation}
where the product runs over the set of maximal points $\eta$ of $S$. The following statement is due to Raynaud.

\begin{proposition}\label{ray1} Let $f\colon X\to S$ be a faithfully flat morphism of locally noetherian schemes which is either quasi-compact or locally of finite type. Assume that $S$ is normal and that, for every point $s$ of $S$ of codimension $1$, the fiber $X_{\lbe s}$ is integral. Then the canonical sequence of abelian groups
\[
\pic S\overset{\!\be\pic\be f}{\lra}\pic X\overset{\!(\pic\lbe j_{\eta}\lbe)}{\lra}\prod\pic X_{\eta}
\]
is exact.
\end{proposition}
\begin{proof} See \cite[${\rm Err}_{\e\rm IV}$, 53, Corollary 21.4.13, p.~361]{ega}.
\end{proof}

Recall that a (locally noetherian) scheme $X$ is called {\it locally factorial} if, for every $x\in X$, the local ring $\s O_{\be X,\le x}$ is factorial. The following implications hold for locally noetherian schemes: regular $\implies$ locally factorial $\implies$ normal.

\begin{corollary}\label{ray2} Let the notation and hypotheses be those of the proposition. Then the map \eqref{pico} induces an injection $\npic\be(X\be/\be S\le)\hookrightarrow \prod\pic X_{\eta}$, where $\npic\be(X\be/\be S\le)$ is the group \eqref{npic}. If, in addition, $X$ is locally factorial, then the preceding map is an isomorphism, i.e., $\npic\be(X\be/\be S\le)\overset{\!\sim}{\to}\prod\pic X_{\eta}$.
\end{corollary}
\begin{proof} This is immediate from the proposition except for the last assertion, i.e., the surjectivity of the map \eqref{pico} when $X$ is locally factorial. If $\{S_{\alpha}\}$ is the family of irreducible components of $S$, then \eqref{pico} is the product of the maps $\pic\be(X\!\times_{S}\! S_{\alpha})\to \pic X_{\eta(\alpha)}$, where $\eta(\alpha)$ denotes the generic point of $S_{\alpha}$. Thus we may assume that $S$ is irreducible with generic point $\eta$. 
By \cite[${\rm IV}_{\lbe 4}$, Corollary 21.6.10(ii)]{ega}, the map $\pic X\to\pic X_{\eta}$ can be identified with the map of divisor class groups ${\frak{Cl}}\, X\to {\frak{Cl}}\, X_{\eta}$. Thus it suffices to check that every closed and irreducible subscheme $D_{\eta}$ of codimension 1 in $X_{\eta}$ extends to a closed and irreducible subscheme $D$ of codimension 1 in $X$. Since $\eta\to S$ is quasi-compact, the canonical morphism $D_{\eta}\to X$ is quasi-compact as well and the schematic closure $D$ of $D_{\eta}$ in $X$ is defined \cite[Corollary 6.10.6, p.~325]{ega1}. Since $D$ is closed and irreducible of codimension 1 in $X$, the proof is complete.
\end{proof}

\begin{proposition}\label{opo} Let $S$ be a locally noetherian normal scheme and let $f\colon X\to S$ and $g\colon Y\to S$ be faithfully flat morphisms locally of finite type. Assume that, for every point $s\in S$ of codimension $\leq 1$, the fibers $X_{\lbe s}$ and $Y_{\lbe s}$ are geometrically integral. Then the canonical map 
\[
\npic\be(X\be/\be S\le)\oplus\npic\be(Y\be/\be S\le)\to\npic\be(X\!\times_{\be S}\be Y\be/\be S\le)
\]
\eqref{pxy1b} is injective. If, in addition, $X$, $Y$ and $X\times_{S}Y$ are locally factorial, then there exists a canonical isomorphism
\[
\frac{\npic\be(X\!\times_{\be S}\be Y\be/\be S\le)}{\npic\be(X\be/\be S\le)\oplus\npic\be(Y\be/\be S\le)}\simeq
\prod\frac{\pic\be(X_{\eta}\be\times_{k(\eta)}\be Y_{\eta}\le)}{\pic X_{\eta}\lbe\oplus\lbe \pic Y_{\eta}}.
\]
\end{proposition}
\begin{proof} There exists a canonical commutative diagram of abelian groups
\[
\xymatrix{0\ar[r] &\npic\be(X\be/\be S\le)\oplus\npic\be(Y\be/\be S\le)\ar[r]\ar[d]&\prod\left(\pic X_{\eta}\be\oplus\lbe \pic Y_{\eta}\right)\ar@{^{(}->}[d]\\
0\ar[r]&\npic\be(X\!\times_{\be S}\be Y\be/\be S\le)\ar[r]&\prod\pic\be(X_{\eta}\!\times_{\eta}\! Y_{\eta}\le)},
\]
where the products run over the set of maximal points $\eta$ of $S$, the top row is exact by Corollary \ref{ray2} and the right-hand vertical map is injective by Proposition \ref{inj}. If $s\in S$ is a point of codimension 1, then $(X\be\times_{\be S}\lbe Y\le)_{s}=X_{\lbe s}\be\times_{\lbe k(s)}\lbe Y_{\lbe s}$ is geometrically integral by
\cite[${\rm IV}_{2}$, Proposition 4.6.5(ii)]{ega}. Thus the bottom row in the above diagram is exact as well by Corollary \ref{ray2}, which yields the first assertion of the proposition. If $X$, $Y$ and $X\times_{S}Y$ are locally factorial, then Corollary \ref{ray2} shows that both nontrivial horizontal maps in the above diagram are isomorphisms, whence the second assertion of the proposition is clear.
\end{proof}

The next corollary describes the kernel of the map \eqref{pxy1} under certain conditions.

\begin{corollary}\label{non}  Let $S$ be a locally noetherian normal scheme and let $f\colon X\to S$ and $g\colon Y\to S$ be faithfully flat morphisms locally of finite type. Assume that, for every point $s\in S$ of codimension $\leq 1$, the fibers $X_{\lbe s}$ and $Y_{\lbe s}$ are geometrically integral. Assume, furthermore, that $I\lbe(\le f\!\times_{\be S}\! g\le)$ is defined and is equal to $1$. Then there exists a canonical exact sequence of abelian groups
\[
0\to\pic S\overset{\!\beta^{1}}{\lra}\pic X\be\oplus\be\pic Y\overset{\!p_{\lbe XY}^{1}}{\lra} \pic\be(X\!\times_{\be S}\! Y\e),
\]
where the maps $\beta^{1}$ and $p_{\lbe XY}^{1}$ are given by \eqref{beta} and \eqref{pxy}, respectively.
\end{corollary}
\begin{proof} This follows by combining the proposition and Corollary \ref{cop0}.	
\end{proof}

The following statement generalizes a theorem of Ischebeck \cite[Theorem 1.7]{isch}, who obtained the conclusion below when the field $k$ is algebraically closed.

\begin{theorem}\label{isc} Let $k$ be a field and let $X$ and $Y$ be normal and geometrically integral $k$-schemes locally of finite type. Then there exists a canonical exact sequence of abelian groups
\[
0\to\pic X\!\oplus\be \pic Y\to \pic\be(X\!\times_{k}\!Y\le)\to\pic\be(\le k\lle(\lbe X)\be\otimes_{\e k}\be k\lle(\le Y\le)),
\]
where $k(X)$ and $k(\le Y\le)$ are the function fields of $X$ and $Y$, respectively.
\end{theorem}
\begin{proof} By Proposition \ref{inj}, we need only check exactness at $\pic\be(X\!\times_{\lbe k}\be Y\le)$. Since $Y_{k(x)}$ is integral for every point $x\in X$ of codimension 1, Proposition \ref{ray1} yields an exact sequence of abelian groups
\[
\pic X\to\pic\be(X\!\times_{k}\!Y\le)\to\pic\be(\e \spec k\lbe(\lbe X)\times_{k}Y\e).
\]
On the other hand, since $k(\be X)/k$ is separable and primary \cite[${\rm IV}_{2}$, Corollary 4.6.3]{ega}, $k(X)\otimes_{k}k(y)$ is integral for every point $y\in Y$ of codimension 1
\cite[${\rm IV}_{2}$, Propositions 4.3.2 and 4.3.5]{ega}. Thus Proposition \ref{ray1} yields an exact sequence of abelian groups
\[
\pic Y\to\pic\be(\e \spec k\lbe(\lbe X)\times_{k}Y\e)\to\pic\be(\le k(\lbe X)\be\otimes_{\e k}\be k(Y\le)).
\]
The exactness of the sequence of the theorem at $\pic\be(X\be\times_{\le k}\be Y\le)$ now follows as in \cite[proof of Theorem 1.7, p.~144]{isch}, i.e., by considering the exact and commutative diagram of abelian groups
\[
\xymatrix{&&\pic Y\ar[dl]\ar[d]\\
\pic X\ar[r]&\pic\be(X\!\times_{k}\!Y\le)\ar[r]\ar[dr]&\pic\be(\e \spec k\lbe(\lbe X)\times_{k}Y\e)\ar[d]\\
&& \pic\be(\le k\lbe(\lbe X)\be\otimes_{\e k}\be k\lbe(Y))
}
\]
whose row (respectively, colummn) is the first (respectively, second) exact sequence above.
\end{proof}

Recall that, if $k^{\e\prime}\be/k$ is a field extension, a geometrically integral $k$-scheme $X$ is called {\it $k^{\e\prime}$-rational} if the function field of $X\be\times_{k}\be\spec k^{\e\prime}$ is a purely transcendental extension of $k^{\e\prime}$. The following corollary of the theorem generalizes 
\cite[Lemma 11, p.~188]{cts77} (mainly by replacing the smoothness hypothesis in [loc.cit.] by a normality hypothesis).

\begin{corollary} \label{isc1} Let $k$ be a field and let $X$ and $Y$ be normal and geometrically integral $k$-schemes locally of finite type. Assume that either $X$ or $Y$ is $k$-rational. Then the canonical map $\pic X\oplus\le \pic Y\to\pic\be(X\be\times_{k}\be Y\le)$ is an isomorphism.
\end{corollary}
\begin{proof} The rationality hypothesis implies that $k\lbe(\lbe X)\be\otimes_{\e k}\be k\lbe(Y)$ is a ring of fractions of a polynomial ring over a field. Thus, by \cite[VII, \S3, 4 and 5]{bou}, $k\lbe(\lbe X)\be\otimes_{\e k}\be k\lbe(Y)$ is a factorial domain, whence $\pic\be(k\lbe(\lbe X)\be\otimes_{\e k}\be k\lbe(Y))=0$. The corollary is now immediate from the theorem.
\end{proof}

The next proposition generalizes \cite[Lemma 6.6(i), p.~40]{san}.

\begin{proposition}\label{rat} Let $k$ be a field with separable closure $\ks$ and corresponding absolute Galois group $\g={\rm Gal}(\ks\be/k)$ and let $X$ and $Y$ be geometrically connected and geometrically reduced $k$-schemes locally of finite type. Assume that 
\begin{enumerate}
\item[(i)] the canonical map $(\e\pic\xs)^{\g}\be\oplus\be(\e\pic\ys)^{\g}\!\to \pic\be(\xs\!\times_{k^{\e\rm s}}\!\ys\le)^{\g}$ is an isomorphism, and 
\item[(ii)] ${\rm gcd}\le(\le I(X\le),I(\le Y\le))=1$.
\end{enumerate}
Then the canonical map $\pic X\oplus \pic Y\to\pic\be(X\lbe\times_{k}\lbe Y\le)$ is an isomorphism.
\end{proposition}
\begin{proof} By Proposition \ref{inj}, the indicated map is injective. Now, if $Y(k)\neq\emptyset$, then the proof of \cite[Lemma 6.6, p.~40]{san} (which relies on Corollary \ref{kres}\, for $S=\spec k$) shows that the canonical map
\begin{equation}\label{map1}
\frac{\pic\be(X\!\times_{k}\! Y\le)}{\pic X\be\oplus\be \pic Y}\to
\frac{\pic\be(\e\xs\be\times_{\ks}\be\ys\le)^{\g}}{(\e\pic\xs)^{\g}\be\oplus\be (\pic\ys)^{\g}}
\end{equation}
is injective. When $Y(k)=\emptyset$, a restriction-corestriction argument similar to that used in the proof of Proposition \ref{kad} shows that the kernel of the preceding map is annihilated by $I(\le Y\le)$. Interchanging $X$ and $Y$, we conclude that the kernel of \eqref{map1} is annihilated by ${\rm gcd}(I(\le X\le),I(\le Y\le))$, whence the proposition follows.
\end{proof}

\smallskip

\begin{examples}\label{sav1} If $X$ and $Y$ are arbitrary $k$-schemes such that $(X\times_{k} Y\le)(\ks)\neq\emptyset$, then Proposition \ref{inj} yields an injection of $\g$-modules $\pic\xs\be\oplus\be \pic\ys\hookrightarrow \pic\be(\e\xs\be\times_{\ks}\be\ys\le)$ which induces an isomorphism of abelian groups
\begin{equation}\label{gen}
\frac{\pic\be(\e\xs\be\times_{\ks}\be\ys\le)^{\g}}{(\e\pic\xs)^{\g}\be\oplus\be (\pic\ys)^{\g}}\simeq\krn\!\!\left[\left(\frac{\pic\be(\e\xs\!\times_{\ks}\!\ys\le)}{\pic\xs\be\oplus\be \pic\ys}\right)^{\!\!\g}\!\to H^{1}(k,\pic\xs)\oplus H^{1}(k,\pic\ys)\right],
\end{equation}
where the indicated map is a connecting homomorphism in Galois cohomology. Thus hypothesis (i) of the proposition holds if  $(\pic\be(\e\xs\be\times_{\ks}\be\ys\le)/\pic\xs\oplus \pic\ys)^{\g}=0$. This is the case if:
\begin{enumerate}
\item[(a)] $X$ and $Y$ are normal and geometrically integral $k$-schemes locally of finite type and either $\xs$ or $\ys$ is rational. Indeed,  
$\xs$ and $\ys$ are then normal and geometrically integral $\ks$-schemes locally of finite type \cite[${\rm IV}_{2}$, Proposition 6.7.4]{ega} and Corollary \ref{isc1}\, shows that $\pic\be(\e\xs\!\times_{\ks}\!\ys\le)=\pic\xs\be\oplus\be \pic\ys$.

\item[(b)] $k$ is perfect, $Y$ is a (possibly non-reduced) geometrically connected $k$-scheme of finite type and $X$ is a projective $k$-variety such that $H^{1}(\xs,\s O_{\be X^{\rm s}})=0$. Indeed, $\pic\be(\e\xs\!\times_{\ks}\!\ys\le)=\pic\xs\be\oplus\be \pic\ys$ by \cite[Exercise III.12.6(b), p.~292]{hart}.
\end{enumerate}
\end{examples}

\begin{proposition} \label{nice0} Let $k$ be a field and let $X$ and $Y$ be smooth and projective $k$-varieties. Assume that 
\begin{enumerate}
\item[(i)] $\Hom_{\le\lle k}(A,B\le)=0$, where $A$ and $B$ are the Picard varieties of $X$ and $Y$, respectively, and 
\item[(ii)] either $\br\e k=0$ or ${\rm gcd}\le(\le I(X\le),I(\le Y\le))=1$.
\end{enumerate}
Then $\pic\be(X\be\times_{k}\be Y\le)=\pic X\oplus \pic Y$.
\end{proposition}

The following is an immediate consequence of the proposition.

\begin{corollary} \label{ccor} Let $k$ be a field and let $A$ and $B$ be abelian varieties over $k$  such that $\Hom_{\le\lle k}(A,B\le)=0$. Then $\pic\be(A\be\times_{k}\be B\le)=\pic A\oplus\le \pic B$.\qed
\end{corollary}

\begin{proof} (Of Proposition \ref{nice0}) By \cite[Proposition 1.7]{sz}, Galois descent of morphisms of schemes \cite[\S4]{gad} and \cite[Theorem 12.5, p.~123]{mi2}, there exists a canonical isomorphism of free abelian groups of finite $\Z\le$\e-rank $\leq 4\,({\rm dim}\,A) \,({\rm dim}\, B\le)$
\begin{equation}\label{gen2}
\left(\frac{\pic\be(\e\xs\!\times_{\ks}\!\ys\le)}{\pic\xs\be\oplus\be \pic\ys}\right)^{\!\be\g}\isoto\Hom_{\e\ks}((\lbe B^{\e\rm s})^{\lbe\vee}\be,\le A^{\le\rm s}\le)^{\g}=\Hom_{\e k}(B^{\vee}\be,\le A\le),
\end{equation}
where $B^{\vee}$ denotes the dual abelian variety of $B$. It is not difficult to check, using duality and the existence of the canonical isogeny $B\to B^{\vee}$, that $\Hom_{\e k}(B^{\vee}\be,\le A\le)=0$ if, and only if, $\Hom_{\e k}(B,\le A\le)=0$ or, equivalently, if, and only if, $\Hom_{\e k}(A,B\le)=0$. Now \eqref{map1}, \eqref{gen} and \eqref{gen2} together show that there exists a canonical homomorphism of abelian groups
\begin{equation}\label{dad}
\frac{\pic\be(X\!\times_{k}\! Y\le)}{\pic X\be\oplus\be \pic Y}\to\krn[\,\Hom_{\e k}(B^{\vee}\be,\le A\le)\!\to H^{1}(k,\pic\xs)\oplus H^{1}(k,\pic\ys)\e]
\end{equation}
whose kernel is annihilated by ${\rm gcd}\le(\le I(X\le),I(\le Y\le))$. The second arrow in \eqref{dad} is the composition of the inverse of \eqref{gen2} and a connecting homomorphism in Galois cohomology. Note that, since $H^{1}(k,\pic\xs)\oplus H^{1}(k,\pic\ys)$ is a torsion abelian group, the right-hand group in \eqref{dad} is free of the same $\Z\le$\e-rank  as $\Hom_{\e k}(B^{\vee}\be,\le A\le)$.  We will now show that \eqref{dad} is injective if $\br\e k=0$, which will complete the proof.	Since $\ks[\lle X\le]=\ks$ by \cite[Exercise 3.11, pp.~25 and 64]{klei}, \cite[Lemma 6.3(i)]{san} yields a canonical exact and commutative diagram of abelian groups
\[
\xymatrix{0\ar[r]&\pic X\oplus \pic Y\ar@{^{(}->}[d]\ar[r]&(\e\pic\xs)^{\g}\be\oplus\be (\pic\ys)^{\g}\ar@{^{(}->}[d]\ar[r]& \br\e k\oplus\br\e k\ar[d]^(.45){(\e\cdot\e)}\\
0\ar[r]&\pic\be(X\be\times_{k}\be Y\le)\ar[r]& \pic\be(\e\xs\be\times_{\ks}\be\ys\le)^{\g}\ar[r]& \br\e k,
}
\]
where the first two vertical maps are injective by Proposition \ref{inj} and the third vertical arrow is the multiplication map. Note that the kernel of the latter map is canonically isomorphic to $\br\e k$. The diagram shows that the kernel of \eqref{map1}, which is the same as the kernel of \eqref{dad}, is canonically isomorphic to a subgroup of $\br\e k$, which completes the proof.
\end{proof}

\begin{examples}\label{sav2} Let $k,X,Y,A$ and $B$ be as in the statement of Proposition \ref{nice0}. 
\begin{enumerate}
\item[(a)] Since ${\rm dim}\, A\leq {\rm dim}_{\e\ks}\le H^{1}(\xs,\s O_{\be \xs})$ by \cite[Theorem 2.10(iii)]{fga}, and similarly for $\ys$,
condition (i) of Proposition \ref{nice0} clearly holds if either
$H^{1}(\xs,\s O_{\be \xs})=0$ or $H^{1}(\ys,\s O_{ Y^{\lle\rm s}})=0$. Note that there exist $k$-varieties $X$ as above such that $\xs$ is {\it not} rational and $H^{1}(\xs,\s O_{\be \xs})=0$. For example, if $k$ is perfect and $\xs$ is separably rationally connected (e.g., rationally connected \cite[Theorem IV.3.9, p.~203]{k}), then $H^{1}(\xs,\s O_{\be \xs})=0$ \cite{gou}.

\item[(b)] The works of Zarhin \cite{zar1}, \cite[Theorems 2.1 and 3.12]{zar2} and \cite[Theorem 1.1]{zar3} contain numerous examples of nonzero abelian varieties $A$ and $B$ such that $\Hom_{\e\ks}\lbe(A^{\le\rm s}\lbe, B^{\e\rm s}\le)=0$ (and thus $\Hom_{\le\lle k}(A,B\le)=\Hom_{\e\ks}\lbe(A^{\le\rm s}\lbe, B^{\e\rm s}\le)^{\g}=0$). We select from \cite[p.~14]{zar1} the following example. For every integer $n\geq 3$, let $C_{n}$ denote the (smooth projective) hyperelliptic curve over $\Q$ with affine model $y^{\le 2}=x^{n}-x-1$ and let $J(C_{n})$ denote the Jacobian variety of $C_{n}$. If $n\geq 3,m\geq 5$ and $n\neq m$, then $\Hom_{\le\lle \Q}(J(C_{n}),J(C_{m})\le)=0$. Thus if, in addition, either $n$ or $m$ is even (so that either $C_{n}(\Q)\neq\emptyset$ or $C_{m}(\Q)\neq\emptyset$), then Proposition \ref{nice0} shows that $\pic\be(C_{n}\be\times_{\Q}\be C_{m}\le)=\pic C_{n}\oplus \pic C_{m}$.

\item[(c)] Certainly, there exist examples where $\Hom_{\e\ks}\lbe(A^{\le\rm s}\lbe, B^{\e\rm s}\le)\neq 0$ but $\Hom_{\le\lle k}(A,B\le)=\Hom_{\e\ks}\lbe(A^{\le\rm s}\lbe, B^{\e\rm s}\le)^{\g}=0$. The following example is due to W. Sawin \cite{saw}. Let $A$ be an abelian variety over $k$ such that ${\rm End}_{\e\ks}\lbe(A^{\le\rm s})=\Z$ and let $B=A^{\chi}$ be a quadratic twist of $A$. Then there exists a canonical isomorphism of abelian groups $\Hom_{\e\ks}\lbe(A^{\le\rm s}\lbe, B^{\e\rm s}\le)\simeq{\rm End}_{\e\ks}\lbe(A^{\le\rm s})=\Z$. Under the above isomorphism, the action of $\g$ on $\Hom_{\e\ks}\lbe(A^{\le\rm s}\lbe, B^{\e\rm s}\le)$ corresponds to the action of $\g$ on $\Z$ via $\chi$. Thus $\Hom_{\le\lle k}(A,B\le)=\Hom_{\e\ks}\lbe(A^{\le\rm s}\lbe, B^{\e\rm s}\le)^{\g}=0$.
Now Corollary \ref{ccor} yields the following statement: if ${\rm End}_{\e\ks}\lbe(A^{\le\rm s})=\Z$ and $A^{\chi}$ is a quadratic twist of $A$, then $\pic\be(A\be\times_{k}\be A^{\chi}\le)=\pic\be A\oplus \pic\lbe A^{\chi}$.
\end{enumerate}
\end{examples}

\begin{theorem}\label{non2}  Let $S$ be a locally noetherian normal scheme and let $f\colon X\to S$ and $g\colon Y\to S$ be faithfully flat morphisms locally of finite type. Assume that
\begin{enumerate}
\item[(i)]  $X$, $Y$ and $X\be\times_{S}\be Y$ are locally factorial,
\item[(ii)] for every point $s\in S$ of codimension $\leq 1$, the fibers $X_{\lbe s}$ and $Y_{\be s}$ are geometrically integral and
\item[(iii)] the \'etale index $I(\le f\!\times_{\be S}\be g\le)$ is defined and is equal to $1$.
\end{enumerate}	
Then there exists a canonical exact sequence of abelian groups
\[
0\to\pic S\overset{\!\beta^{\lle 1}}{\lra}\pic X\be\oplus\be\pic Y\overset{\!p_{\lbe XY}^{\lle 1}}{\lra} \pic\be(X\!\times_{\be S}\! Y\e)\to\prod\frac{\pic\be(X_{\eta}\be\times_{k(\eta)}\be Y_{\eta}\le)}{\pic X_{\eta}\lbe\oplus\lbe \pic Y_{\eta}}\to 0,
\]
where the maps $\beta^{1}$ and $p_{\lbe XY}^{1}$ are given by \eqref{beta} and \eqref{pxy}, respectively, and the product runs over the set of maximal points $\eta$ of $S$.
\end{theorem}
\begin{proof} By Corollary \ref{non} the indicated sequence is exact at the first two nontrivial terms and by Proposition \ref{eqcor} the cokernel of the canonical map $\pic X\be\oplus\be\pic Y\to \pic\be(X\!\times_{\be S}\be Y\e)$ is naturally isomorphic to 
$\npic\be(\le  X\!\times_{\be S}\!  Y\be/\be S\le)/\e	\npic\be(\le X\be/\be S\le)\oplus\npic\be(\le Y\be/\be S\le)$. Now Proposition \ref{opo} completes the proof.
\end{proof}

\begin{corollary}\label{non3} Let the notation and hypotheses be as in the theorem. Assume, in addition, that $\pic\be(X_{\eta}^{\le\rm s}\be\times_{k(\eta)^{\rm s}}\be Y_{\!\eta}^{\e\rm s}\le)^{\g(\eta)}=(\e\pic X_{\eta}^{\le\rm s})^{\g(\eta)}\be\oplus\be(\e\pic Y_{\!\eta}^{\e\rm s})^{\g(\eta)}$ for every maximal point $\eta$ of $S$. Then there exists a canonical exact sequence of abelian groups
\[
0\to\pic S\to\pic X\be\oplus\be\pic Y\to \pic\be(X\!\times_{\be S}\! Y\e)\to 0.
\]
\end{corollary}
\begin{proof} This follows from the theorem and Proposition \ref{rat} noting that, by \eqref{genf}, hypothesis (iii) of the theorem implies hypothesis (ii) of Proposition \ref{rat}.
\end{proof}

\end{document}